\documentclass[a4paper,12pt]{amsart}
\usepackage{amsmath,amsfonts,amssymb,amsthm}

\usepackage{pstricks}
\usepackage{subfig}
\usepackage{graphicx}
\usepackage{color}
\usepackage[all]{xy}
\usepackage{url}
\usepackage{setspace}
\usepackage{booktabs}
\usepackage{longtable}
\usepackage{latexsym}

\newtheorem{theorem}{Theorem}
\newtheorem{lemma}[theorem]{Lemma}
\newtheorem{proposition}[theorem]{Proposition}
\newtheorem{corollary}[theorem]{Corollary}

\theoremstyle{definition}
\newtheorem{definition}[theorem]{Definition}

\newtheorem{remark}[theorem]{Remark}

\DeclareMathOperator{\innt}{int}

\DeclareMathOperator{\spec}{Spec}
\DeclareMathOperator{\Spec}{Spec}

\newcommand{\ind}[0]{\operatorname{index}}

\DeclareMathOperator{\conv}{conv}

\newcommand{\kG}{\Gamma}

\newcommand{\idm}{\mathfrak{m}}

\newcommand{\keps}{\varepsilon}

\newcommand{\A}{\mathbb A}

\newcommand{\N}{\mathbb N}

\newcommand{\T}{\mathbb T}
\newcommand{\V}{\mathbb V}
\newcommand{\Z}{\mathbb Z}

\newcommand{\Q}{\mathbb Q}

\newcommand{\CF}{{\mathcal F}}
\newcommand{\CG}{{\mathcal G}}

\newcommand{\CO}{{\mathcal O}}

\newcommand{\CR}{{\mathcal R}}

\definecolor{intOrange}{rgb}{1.0,.310,.0} 
                   
\newcommand{\til}[1]{\widetilde{#1}}
\renewcommand{\iff}{\Leftrightarrow}

\newcommand{\then}{\Rightarrow}
\newcommand{\theen}{\Longrightarrow}
\newcommand{\neht}{\Leftarrow}

\newcommand{\lHom}{Hom}

\newcommand{\largestint}[1]{\lfloor #1 \rfloor}
\newcommand{\rounddown}[1]{\largestint{#1}}

\newcommand{\Sl}{\operatorname{SL}}

\newcommand{\spann}{\operatorname{span}}

\newcommand{\kst}{\,|\;}

\newcommand{\surj}{\rightarrow\hspace{-0.8em}\rightarrow}

\newcommand{\jrus}{\leftarrow\hspace{-0.8em}\leftarrow}

\newcommand{\kss}{\scriptscriptstyle}
\newcommand{\kbb}{{\kss \bullet}}
\newcommand{\ko}{\overline}

\newcommand{\dual}{^{\scriptscriptstyle\vee}}

\newcommand{\kk}{k}

\newcommand{\Matc}[2]{\left(\begin{array}{@{}*{#1}{c}@{}} #2
\end{array}\right)}

\newcommand{\siga}{\alpha}
\newcommand{\sigb}{\beta}
\newcommand{\sigc}{\gamma}
\newcommand{\sigr}{r^1}
\newcommand{\sigs}{r^e}
\newcommand{\Ia}{A}
\newcommand{\Ib}{B}
\newcommand{\ci}{\nu}  
\newcommand{\toric}{\T\V}  
\newcommand{\gH}{\operatorname{H}}
\newcommand{\Pic}{\operatorname{Pic}}
\newcommand{\kiso}[1]{\mbox{$(*)_#1$}}
\newcommand{\kISO}[1]{\mbox{$\ko{(*)}_#1$}}
\newcommand{\kr}{g}   
\newcommand{\km}{m}   
\newcommand{\kX}{S}  
\newcommand{\kA}{A}  
\newcommand{\kB}{B}  
\newcommand{\kra}{s}
\newcommand{\krb}{t}

\newcommand{\kT}{{\T}}

\newcommand{\z}[0]{{\mathbb Z}}
\newcommand{\res}[0]{\operatorname{\mathcal R}}
\newcommand{\qg}{\operatorname{qG}}
\newcommand{\qG}{\mbox{qG}}
\newcommand{\VW}{\mbox{V\!W}}
\newcommand{\vw}{\operatorname{V\!W}}
\renewcommand{\a}[0]{{\mathbb A}}

\newcommand{\embdim}[0]{\operatorname{embdim}}

\newcommand{\chr}[0]{\operatorname{char}}
\newcommand{\TT}{{T^1}}
\newcommand{\TTqG}{{T^1_{\qg}}}
\newcommand{\TTV}{{T^1_V}}
\newcommand{\TTW}{{T^1_W}}
\newcommand{\TTVW}{{T^1_{\vw}}}

\begin{document}

\title[The dualizing sheaf on deformations of   
toric surface singularities]
{The dualizing sheaf on first-order deformations of 
toric surface singularities}

\author[K.~Altmann]{Klaus Altmann
}
\address{Institut f\"ur Mathematik,
FU Berlin,
Arnimalle 3,
14195 Berlin,
Germany}
\email{altmann@math.fu-berlin.de}
\author[J.~Koll\'{a}r]{J\'{a}nos~Koll\'{a}r}
\address{Department of Mathematics,
Princeton University,
Fine Hall, 
Washington Road Princeton, NJ 08544-1000, USA}
\email{kollar@math.princeton.edu}
\thanks{Partial financial support to KA was provided by the 
DFG via the CRC 647 and to
JK
by the NSF under grant number DMS-1362960.}

\begin{abstract}
We explicitly describe infinitesimal deformations of cyclic quotient
singularities that satisfy one of the deformation conditions introduced by
Wahl,
Koll\'ar--Shepherd-Barron  and Viehweg. The conclusion is that in many
cases
these three notions are different from each other. In particular, we see that
while the KSB and the Viehweg versions of the moduli space of surfaces of
general type have the same underlying reduced subscheme, their
infinitesimal structures are different.
\end{abstract}

\maketitle

\section{Introduction}
\label{intro}

In order to compactify the moduli space of surfaces of general type,
one has to consider singular surfaces but for a long time it was not clear 
which class of singularities should be allowed. Building on Mori's 
program,
\cite{KSB} described such a  class, named
semi-log-canonical singularities.
These include quotient singularities, cusps and a few others;
see \cite[Sec.2.2]{kk-singbook} for a complete list.

A new feature of the theory is  that not every flat deformation of a 
surface
with such singularities should be allowed in moduli theory.
In essence this observation can be traced back to Bertini
who observed that the cone over the degree 4 rational normal curve admits 
two distinct smoothings. One is the Veronese surface 
the other is a ruled surface;
see \cite{pin-def}.
For the Veronese the self-intersection of the canonical class is 9 for the 
ruled surface it is 8. Since we would like the basic numerical invariants 
to be locally constant in families, one of these deformations should not 
be allowed.

It is not obvious how to obtain the right class of deformations.
  Three variants  have been
investigated in the past. Their common feature is that they all study
the compatibility of  deformations  with
powers of the dualizing sheaf $\omega$. 
In order to define these 3 versions, we need some definitions.

\subsection{General setup}
\label{QG.V.defn}
We are ultimately interested in schemes with semi-log-ca\-no\-ni\-cal
singularities $\kX$, but for the basic definitions we need to assume only 
that
$\kX$ is a pure dimensional
$S_2$ scheme over a field $k$ such  that
\begin{enumerate}
\item[(i)] there is a closed subset $Z\subset\kX$ of codimension $\geq 2$
such that $\omega_{\kX\setminus Z}$ is locally free and
\item[(ii)] there is an $\km>0$ such that
$\omega_\kX^{[\km]}$ is locally free,
\end{enumerate}
where $\omega_\kX^{[\km]}$  denotes the reflexive hull of
$\omega_\kX^{\otimes \km}$.
The smallest such $\km>0$ is called the {\it index} of $\omega_\kX$.
(Both of these conditions are satisfied by schemes with semi-log-canonical
singularities.)
\\[1ex]
Let $(0,T)$ be a local scheme such that $k(0)\cong k$ and
$p:X_T\to T$ a flat deformation of $\kX\cong X_0$.
For every $\kr\in \z$ we have natural restriction maps
$$
\res^{[\kr]}:\omega_{X_T/T}^{[\kr]}|_{X_0}\to \omega_{X_0}^{[\kr]}.
$$
These maps are isomorphisms over $\kX\setminus Z$ and we are interested in
understanding those cases when they are isomorphisms over $\kX$.
The local criterion of flatness shows  (see \cite{k-modbook} for details)
that if $T$ is Artinian then
$$
\res^{[\kr]} \mbox{ is an isomorphism }
\iff
\res^{[\kr]} \mbox{ is surjective }
\iff
\omega_{X_T/T}^{[\kr]} \mbox{ is flat over } T.
$$
We will denote this condition by $\kiso{\kr}$ (with $\kr\in\Z$).

\subsection{Definitions of \qG- and V- and VW-deformations}
\label{QG.V.defn.2}
Let $p:X_T\to T$  be a flat deformation as in (\ref{QG.V.defn}).

\subsubsection{$\qg$-deformations}
\label{def-qG}
We call $p:X_T\to T$ a  {\it $\qg$-deformation} if the
conditions
\kiso{\kr} defined in (\ref{QG.V.defn})
hold  for every $\kr\in\Z$.
It is enough to check these for  $\kr=1,\dots, \ind(\omega_\kX)$.
($\qg$ is short for ``Quotient of Gorenstein,'' but this is misleading if
$\dim \kX\geq 3$.)
\\[1ex]
These deformations were introduced and studied by
Koll\'ar and Shepherd-Barron \cite{KSB} as the class most suitable for
compactifying the
moduli of varieties of general type. A list of log canonical surface
singularities with $\qg$-smoothings is given in \cite{KSB}.
In the key case of cyclic quotient singularities the list 
(of the so-called T-singularities) was
earlier established by Wahl \cite[2.7]{MR82d:14004}, though he viewed them 
as
examples of W-deformations (see below).

\subsubsection{V-deformations}
\label{def-V}
We call $p:X_T\to T$ a   {\it Viehweg-type deformation}
(or V-deformation) if  the conditions
\kiso{\kr} from (\ref{QG.V.defn})
hold for every $\kr$ divisible by
$\ind(\omega_\kX)$. It is enough to check this for $\kr=\ind(\omega_\kX)$.
\\[1ex]
These deformations form the natural class suggested by the geometric 
invariant theory methods  used in the monograph \cite{vieh-book}.
Actually,
\cite{vieh-book} considers the---a priori
weaker---condition: $\res^{[\kr]}$ is an isomorphism
for some $\kr>0$ divisible by
$\ind(\omega_\kX)$. One can see that in this case
\kiso{\kr}
holds for every $\kr$ divisible by $\ind(\omega_\kX)$, at least
in characteristic 0; see \cite{k-modbook}.
V-deformations are problematic in
positive characteristic, see \cite[14.7]{hac-kov}.

\subsubsection{W-deformations}
\label{def-W}
We call $p:X_T\to T$ a   {\it Wahl-type deformation} (or
W-deformation) if the condition
\kiso{\kr}
holds for $\kr=-1$.
These deformations were considered in \cite{MR82d:14004, MR83h:14029}
and called $\omega^*$-constant deformations there.

\subsubsection{VW-deformations}
\label{def-VW}
We call $p:X_T\to T$ a VW-deformation if
it is both a V-deformation and a W-deformation.

\subsection{Relations between \qG, V and VW}
\label{rel-qG-VW}
It is clear that every $\qg$-deformation is also a VW-deformation.
Understanding the  precise relationship between
the four classes (\ref{def-qG}) -- (\ref{def-VW})
has been a long standing open problem.
For reduced base spaces we have the following very strong result.

\begin{theorem}
\label{th-VqG}
A flat deformation of a log canonical scheme over  a reduced,
local scheme of characteristic 0 is a
  V-deformation  if and only if it  is a $\qg$-deformation.
\end{theorem}

When $T$ is the spectrum of a DVR,  $\dim \kX=2$
and $\kX$ has quotient singularities, this
was proved in \cite{MR541025} and \cite[14.2]{k-flat}.
If $\dim \kX>2$ and $\kX$ has log terminal singularities,
this is a special case of
inversion of adjunction as proved in \cite[Sec.17]{k-etal} and
the log canonical case similarly follows from
\cite{kawakita} and the normality of log canonical centers
\cite{ambro, fujinobook};
see also \cite[Sec.4.3]{kk-singbook}.
These imply the claim for arbitrary reduced base schemes using
\cite{k-hh}; see \cite{k-modbook} for more details.
\\[1ex]
This raised the possibility that every  V-deformation
of a log-canonical singularity is also a $\qg$-deformation over arbitrary 
base
schemes.
It would be enough to check this for Artinian bases. In this note we focus 
on
first order deformations and prove that these two classes are quite 
different
from each other.

\begin{definition}
\label{def-TqGVW}
Let $\kX$ be a scheme satisfying  the conditions
(\ref{QG.V.defn})(i)-(ii).
Let $\TT(\kX)$ be the (possibly infinite dimensional) $k$-vector space
of deformations of $\kX$ over $\spec k[\epsilon]$.
We denote by $\TTqG(\kX)\subset \TT(\kX)$ the space of first order
$\qg$-deformations, $\TTV(\kX)$  the space of first order
V-deformations, $\TTW(\kX)$  the space of first order
W-deformations, and  $\TTVW(\kX)$  the space of first order
VW-deformations.
\end{definition}

We have obvious inclusions
$$
\TTqG(\kX)\;\subset\; \TTVW(\kX)\;\subset\;
\TTV(\kX), \TTW(\kX)\;\subset\; T^1(\kX),
$$
but the relationship between $\TTV(\kX)$ and $ \TTW(\kX) $
is not clear.

\subsection{The case of cyclic quotient singularities}
\label{resultCQS}
We completely describe first order
V-, \VW- and $\qg$-deformations of two-dimensional
cyclic quotient singularities.
The precise answers are stated
in Sections \ref{V-deformations} and \ref{compJanos}.
The main
conclusion is that V-deformations and \VW-deformations, and even more
V-deformations and \qG-deformations are quite different
over Artinian bases.

\begin{theorem}
\label{VW-GQ.thm.1}
Let $S_{n,q}:=\a^2/\tfrac1{n}(1,q) $ denote the
quotient of $\a^2$ by the cyclic group action generated by
$(x,y)\mapsto (\eta x, \eta^q y)$,
where $\eta$ is a primitive $n$th root of unity. Then,
if $\,q\neq -1$ in $(\Z/n\Z)^*$,
i.e.\
if $\,\embdim\bigl(S_{n,q}\bigr)\geq 4$,
$$
\dim \TTV\bigl(S_{n,q}\bigr)- \dim \TTVW\bigl(S_{n,q}\bigr)=
\embdim\bigl(S_{n,q}\bigr) - 4
\hspace{0.5em}\mbox{or}\hspace{0.5em}
\embdim\bigl(S_{n,q}\bigr) - 5.
$$
In particular, if $\embdim\bigl(S_{n,q}\bigr)\geq 6$ then $S_{n,q}$ has
V-deformations that are not $\vw$-deformations, hence also not
$\qg$-deformations.
\end{theorem}

This is a direct consequence of the more detailed Theorem \ref{th-IAB}.
By contrast, $\qg$-deformations and VW-deformations are quite close to 
each
other, as shown by the next result. This will be proved in 
(\ref{proofIntro}).

\begin{theorem}
\label{V-VW.thm.1}
Let $S_{n,q}$ be as in the previous theorem. Then
\begin{enumerate}
\item If $\gcd(n, q+1)=1$ then
$\TTqG\bigl(S_{n,q}\bigr)=\TTVW\bigl(S_{n,q}\bigr)=\{0\}$.
\item  
If $S_{n,q}$ admits a $\qg$-smoothing then
$\TTqG\bigl(S_{n,q}\bigr)=\TTVW\bigl(S_{n,q}\bigr)$.
\item In general  $\,\dim \TTqG\bigl(S_{n,q}\bigr)\leq \dim
\TTVW\bigl(S_{n,q}\bigr)\leq \dim \TTqG\bigl(S_{n,q}\bigr)+1$.
\end{enumerate}
\end{theorem}

\subsection{Using the interval language}
\label{usingInt}
Besides the description of cyclic quotient singularities
in terms of the invariants $n$ and $q$,
there is an alternative possibility by using rational intervals
$I=[-\Ia,\Ib]\subseteq \Q$ with uniform denominators at the end points,
i.e.\ $\Ia$ and $\Ib$ have the same denominator in reduced form.
We call $I$ or the resulting singularity $S_I$ {\em grounded} if
$I$ contains an integer in its interior.
Since integral shifts of $I$ will be neglected, this leads to
$\Ia,\Ib>0$.
See (\ref{coneInterval}) and (\ref{groundedCones}) for details.
This language allows a much more detailed description
of the situation:

\begin{theorem}
\label{th-IAB}
Assume that
$\,\embdim\bigl(S_I\bigr)\geq 4$. Then
\\[0.5ex]
{\bf 1)} If the interval $I$ is not grounded, then the associated
surface singularity $S_I$
has neither \qG- nor \VW-deformations. The dimension of
$\TTV(S_I)$ is $\embdim(S_I)-4$.
\\[0.5ex]
{\bf 2)} If $\Ia,\Ib>0$, then
$\dim \TTV(S_I)=\embdim(S_I)-4 + \rounddown{\Ia}+\rounddown{\Ib}$.
\\[0.5ex]
{\bf 3)} If $\Ia,\Ib>0$ with fractional parts $\{\Ia\}=\frac{1}{\km}$
or $\{\Ib\}=\frac{1}{\km}$, then
$$
\dim \TTVW(S_I)=\dim \TTqG(S_I)=\rounddown{\Ia+\Ib}.
$$
{\bf 4)} If $\Ia,\Ib>0$ with both fractional parts $\{\Ia\}$ and $\{\Ib\}$
different from $\frac{1}{\km}$, then
$$
\dim \TTVW(S_I)=\rounddown{\Ia}+\rounddown{\Ib}+1
\hspace{0.8em}\mbox{and}\hspace{0.8em}
\dim \TTqG(S_I)=\rounddown{\Ia+\Ib}.
$$
\end{theorem}

\begin{proof}
The first two parts, i.e.\ the description of the V-deformations
follows from (\ref{restrictPhi}). The remaining two parts are just another
formulation of Theorem \ref{th-lastDef}.
\end{proof}

\subsection{Implications for moduli spaces}
One can construct compactified moduli spaces for surfaces of general type 
using either KSB-deformations or V-deformations. Let us denote these by
${\mathbf M}(\mbox{KSB})$ and ${\mathbf M}(\mbox{V})$. By Theorem 
\ref{th-VqG},
the underlying reduced structures of these moduli spaces are isomorphic.
As a consequence of our computations we can say that the scheme structures 
are not isomorphic.

More generally, let $X$ be a projective variety with isolated 
singularities
$x_1,\dots, x_m$. Any flat deformation of $X$ restricts to a deformation 
of the singularities  $(x_i, X)$. This induces a map of the local 
deformation spaces
$$
\Re: \operatorname{Def}(X)\to
\operatorname{Def}(x_1, X)\times \cdots \times \operatorname{Def}(x_m, X).
$$
A direct consequence of the definition of \qG-deformations given in
\cite{KSB} is that
$\operatorname{Def}_{\qg}(x, X)$ is smooth
  for 2-dimensional quotient singularities. Our computations show that,
by contrast, $\operatorname{Def}_{V}(x, X)$ is usually non-reduced but
$$
\operatorname{red}\bigl(\operatorname{Def}_{V}(x, 
X)\bigr)=\operatorname{Def}_{\qg}(x, X).
$$
We thus expect that if $X$ is a surface with quotient singularities
then 
$\operatorname{Def}_V(X)$ can be non-reduced but
$\operatorname{Def}_{\qg}(X)$ should be smooth.
This is not true in general, but there are many examples when
local-to-global obstructions vanish and  the map $\Re$ is smooth.
The situation is not well understood for surfaces of general type, but
\cite[Prop 3.1]{MR2581246} shows that
local-to-global obstructions vanish for Del~Pezzo surfaces. Thus we obtain 
that if $S$ is a Del~Pezzo surface with quotient singularities then
$\operatorname{Def}_{\qg}(S)$ is smooth but
$\operatorname{Def}_{V}(S)$ is nonreduced as soon as  $S$ has at least 1 
singular point of multiplicity $\geq 5$.

\section{Five descriptions of cyclic quotient singularities}
\label{fiveCQS}
In (\ref{normAct}) -- (\ref{coneInterval})
we present several ways of representing
two-di\-men\-sio\-nal cyclic quotient singularities 
$S=\A^2_k/G$,
i.e.\ those coming from a cyclic group $G$ acting 
on $\A^2_k$.
While most of them are quite classic, the description (\ref{coneInterval})
seems to be not common so far.
At the end, in (\ref{groundedCones}), we introduce the notion
of grounded singularities. In the language of (\ref{coneInterval})
this becomes especially simple.

\subsection{Normalizing the action}
\label{normAct}
Let $G$ denote a cyclic group of order $n$ 
with $\chr k\nmid n$.
Then, by 
\cite[\S2]{Brieskorn},
every linear action of $G$ on $\A^2_k$ is isomorphic
to some action $\tfrac1{n}(1,q)$ generated by
$$
(x,y)\mapsto (\eta x, \eta^q y),
$$
where $q\in(\Z/n\Z)^*$ and
$\eta$ is a primitive $n$-th root of unity.
The corresponding ring of invariants  is $R_{n,q}:=k[x,y]^G$ and the
corresponding quotient singularity is
$$
S_{n,q}:=\a^2/\tfrac1{n}(1,q)=\spec R_{n,q}.
$$
While we work with this affine model, all the results apply to its
localization, Henselisation or  completion at the origin.
We can also choose  $\eta'=\eta^q$ as our primitive $n$-th root of unity.
This shows the isomorphism
$$
S_{n,q}\cong S_{n,q'} \mbox{ where } 
qq'=1 \mbox{ in } (\Z/n\Z)^*.
$$
Note that we can and will choose a representative for $q$
such that $1\leq q\leq n-1$. The case $q=n-1$ encodes the 
$A_{n-1}$-singularities. These are exceptional for many of the subsequent
formulas, so we assume from now on that $q\neq -1$ in $(\Z/n\Z)^*$.

\subsection{The abc notation}
\label{abcNotation}
Here we just rename the invariants $n$ and $q$. Denote
$b:=\gcd(n,q+1)$, $a:=n/b$, and $c:=(q+1)/b$. 
Hence we know that $\gcd(a,c)=1$, 
and $n$ and $q$ can be recovered as 
$n=ab$ and $q=bc-1$.
When using these invariants, we might write
$S_{abc}=\frac{1}{ab}(1,bc-1)$
instead of $S_{n,q}=\frac{1}{n}(1,q)$.
\\[0.5ex]
Note that the case $q=n-1$ which was just excluded
at the end of (\ref{normAct}) 
can be recovered in the abc language as the case $a=1$.
\\[1ex]
The isomorphic singularities $S_{n,q}$ and $S_{n,q'}$ from (\ref{normAct})
share the same $a$ and $b$, i.e.\ $a'=a$ and $b'=b$. This follows from the
fact that $qq'\equiv 1\hspace{-0.3em}\mod n$
implies $qq'\equiv 1\hspace{-0.3em}\mod b$ and that
$q\equiv -1\hspace{-0.3em}\mod b$ becomes then equivalent to
$1\equiv -q'\hspace{-0.3em}\mod b$. The third invariants $c$ and $c'$
differ. However, it is in general not true that they are mutually inverse
within $(\Z/a\Z)^*$. See the discussion at the end of (\ref{coneInterval}).

\subsection{The toric nature of $S_{n,q}$}
\label{toricSnq}
Dealing with toric varieties involves a standardized
language, cf.~\cite{CoxLittleSchenck} for details: Assume that
$N$ and $M$ are mutually dual free abelian groups of 
finite rank;
with $N_\Q$ and $M_\Q$ we denote the associated $\Q$-vector spaces;
similarly we often write $N_\kk$ and $M_\kk$ for $N\otimes_\Z\kk$ and
$M\otimes_\Z\kk$, respectively.
Let $\sigma\subseteq N_\Q$ be a polyhedral cone
and denote by $\sigma\dual:=\{r\in M_\Q\kst \langle \sigma,r\rangle\geq 0\}$
the dual one. Then, $\sigma\dual\cap M$ is a finitely generated semigroup,
and its $M$-graded semigroup ring 
(with $\kk$-basis $\{x^r\kst r\in \sigma\dual\cap M\}$)
provides the affine toric variety
$$
\toric(\sigma,N):=\Spec \kk[\sigma\dual\cap M].
$$
Since we are going to deal with surface singularities,
$N$ and $M$ will be of rank two. Hence, the primitive generators
of $\sigma$ and $\sigma\dual$ are just pairs $\siga,\sigb\in N$ and
$\sigr,\sigs\in M$ 
(with $\langle\siga,\sigr\rangle=\langle\sigb,\sigs\rangle=0$),
respectively. The relation to (\ref{normAct}) is the well-known

\begin{proposition}
\label{prop-toricCQS}
$\,S_{n,q}=\toric(\sigma,\Z^2)$ where
$
\,\sigma=\langle (1,0),\,(-q,n)\rangle\subseteq\Q^2=N_\Q
$.
\end{proposition}

Note that we use $\langle\kbb,\kbb\rangle$ to denote both the pairing
$N\times M\to\Z$ and the generation of a polyhedral cone. Moreover,
when using coordinates, we try to distinguish between $M$ and $N$ by using
the different brackets $[\kbb,\kbb]$ and $(\kbb,\kbb)$, respectively.
So we will write 
$\sigma\dual=\langle [0,1],\,[n,q]\rangle\subseteq\Q^2=M_\Q$.
That is, $\siga=(1,0)$, $\sigb=(-q,n)$, $\sigr=[0,1]$, and $\sigs=[n,q]$.
The group order $n$ may be recovered as $\det\sigma:=\det(\siga,\sigb)$.

\begin{proof}
Writing $k[x,y]$ as the semigroup ring $k[\N^2]$, we know that
$R_{n,q}=k[x,y]^G=k[\N^2\cap M]$ 
where $M\subseteq\Z^2$ is the sublattice
freely generated by, e.g.\ $[-q,1], [n,0]\in\Z^2$.
Now, a linear combination
$
\lambda\cdot[-q,1]+\mu\cdot [n,0]=[-\lambda q + \mu n,\, \lambda]
$
has non-negative entries if and only if
$
\langle (-q,n), [\lambda,\mu]\rangle\geq 0
$
and
$
\langle (1,0), [\lambda,\mu]\rangle\geq 0
$.
\end{proof}

\subsection{Equations of $S_{n,q}$ via continued fractions}
\label{equSnq}
Let $S:=\toric(\sigma)$ for some two-dimensional
cone $\sigma=\langle\siga,\sigb\rangle$ as in (\ref{toricSnq}).
Denote by $E\subset\sigma\dual\cap M$ the set of indecomposable 
elements within this semigroup (``Hilbert basis'').
This finite set coincides with the lattice points on the
compact edges of $\conv(\sigma\dual\cap M\setminus 0)$. In particular, we can
naturally list its elements as 
$E=\{\sigr,r^2,\ldots,r^{e-1},\sigs\}$ with $e\geq 4$
(the cases $e=2$ and $e=3$ refer to $S$ being smooth or an
$A_{n-1}$-singularity).
Any two adjacent elements of this set do always form 
a $\Z$-basis of $M\cong\Z^2$.
Hence, for $i=2,\ldots,e-1$, we can write 
$$
r^{i-1}+r^{i+1}=a_i\cdot r^i
$$ 
with natural numbers $a_i\geq 2$. 
The continued fraction $[a_2,\ldots,a_{e-1}]:=
a_2-\frac{1}{a_3-\ldots}$ recovers
$\frac{n}{n-q}$. Moreover, $E$ provides an embedding
$S\hookrightarrow\A_\kk^e$. Among the equations one finds
$x_{i-1}x_{i+1}-x_i^{a_i}$, see \cite{riems209} for more details.

\begin{remark}
\label{rem-Ji}
With growing $i$, the values $\langle\siga,r^i\rangle$ 
and $\langle\sigb,r^i\rangle$ increase and decrease, respectively.
Hence, defining 
$\eta_i:=\min\{\frac{\langle\siga,r^{i+1}\rangle}{\langle\siga,r^i\rangle},
\frac{\langle\sigb,r^{i-1}\rangle}{\langle\sigb,r^i\rangle}\}\in\Q_{\geq 1}$,
we obtain that $\rounddown{\eta_i}=a_i-1\in\Z_{\geq 1}$.
\end{remark}

\subsection{Replacing cones by intervals}
\label{coneInterval}
Let $\sigma=\langle\siga,\sigb\rangle$ and 
$\sigma\dual=\langle\sigr,\sigs\rangle$ be mutually dual
(two-dimensional, rational) cones as before. The primitive elements
$R\in\innt\sigma\dual\cap M$ (we will call them 
{\em primitive degrees} of $\sigma$) 
give rise to affine crosscuts $Q(\sigma,R):=\sigma\cap[R=1]$.
Since the affine line $[R=1]$ can be identified with 
the rational line $\Q^1$
(canonically, up to integral shifts), we can and will
understand $Q(\sigma,R)$ as an interval in $\Q$.
\\[1ex]
Reciprocally, every closed interval $I\subseteq \Q$ provides a cone via
$C(I):=\Q_{\geq 0}\cdot (I,1)\subseteq\Q^2$ and a primitive degree
$R:=[0,1]$. These two constructions
provide a natural one-one correspondence
$$
\big\{\mbox{pairs $(\sigma,R)$}\big\}\big/\Sl(2,\Z)
\longleftrightarrow
\big\{\mbox{bounded intervals $I\subseteq\Q$}\big\}\big/
\{\mbox{$\Z$-shifts}\}.
$$
On the other hand, every cone $\sigma$ provides a canonical primitive
degree $\ko{R}$,
called the central degree.
It is defined as the
primitive generator of the ray $\Q_{\geq 0}\cdot (\sigr+\sigs)$.
It is the only primitive degree such that 
$\langle\siga,\ko{R}\rangle = \langle \sigb,\ko{R}\rangle$.
Using coordinates via the $(n,q)/(a,b,c)$ language 
discussed in (\ref{abcNotation}),
one obtains that $\sigr+\sigs=[n,q+1]=b\cdot[a,c]$, hence
$\ko{R}=[a,c]=\frac{\sigr+\sigs}{b}$.

\begin{remark}
\label{rem-index}
Actually, $S=\toric(\sigma)$ is $\Q$-Gorenstein with index $a$
(and we suppose that $a>1$).
The corresponding power $\omega_S^{[a]}$ equals the ideal
$(x^{\ko{R}})\subseteq\CO_S$ represented by the shifted semigroup
$\ko{R}+(\innt\sigma\dual\cap M)$. Thus, properly speaking,
not $\ko{R}$ but the non-integral $\frac{1}{a}\ko{R}$ is
the truly canonical degree.
\end{remark}

Using this special central degree $\ko{R}$, 
the previous correspondence yields

\begin{proposition}
\label{prop-CentrDeg}
There is a one-one correspondence
$$
\big\{\mbox{\rm cones $\sigma$}\big\}\big/\Sl(2,\Z)
\longleftrightarrow
\big\{\mbox{\rm intervals $I\subseteq\Q$ with uniform denominators}\big\}\big/
\{\mbox{\rm $\Z$-shifts}\}.
$$
\end{proposition}

We call $I$ to have ``{\em uniform denominators}''
(at the end points) if both become equal in the reduced forms,
i.e.\ if $I=[\frac{g}{\km},\frac{h}{\km}]$ with $g,h,\km\in\Z$ and
$\gcd(g,\km)=\gcd(h,\km)=1$.

\begin{proof}
$(\then$) After a possible coordinate change, we may assume that
$\ko{R}=[0,1]$. Setting 
$\km:=\langle \siga,[0,1]\rangle = \langle \sigb,[0,1]\rangle$
we obtain that $\siga=(g,\km)$ and $\sigb=(h,\km)$, hence
$Q(\sigma,\,[0,1])=[\frac{g}{\km},\frac{h}{\km}]$
for some $g,h$ as asked for in the claim.
\\[1.0ex]
($\neht$) If $I=[\frac{g}{\km},\frac{h}{\km}]$,
then $C(I)=\langle (\frac{g}{\km},1),\,(\frac{h}{\km},1)\rangle
= \langle (g,\km),\,(h,\km)\rangle$, i.e.\ its primitive generators are
$\siga=(g,\km)$ and $\sigb=(h,\km)$. Thus, $R=[0,1]$ coincides with $\ko{R}$.
\end{proof}
\vspace{1ex}

In (\ref{toricSnq}) we had considered cones 
$\sigma=\langle \siga,\sigb\rangle=\langle (1,0),\,(-q,n)\rangle$, i.e.\
$n=|\det(\siga,\sigb)|$, and $q$ was characterized by $n|(q\siga+\sigb)$.
Alternatively we had used $b:=\gcd(n,q+1)$ to write $n=ab$ and $q+1=bc$ 
in (\ref{abcNotation}).
Now, given an interval $I=[\frac{g}{\km},\frac{h}{\km}]$ as in 
Proposition~\ref{prop-CentrDeg}, 
it has length $|I|=\frac{h-g}{\km}$, and
we may obtain the invariants $(a,b,c)$ for $\sigma:=C(I)$ via

\begin{proposition}
\label{prop-abcInv}
For $I=[\frac{g}{\km},\frac{h}{\km}]$ 
\vspace{0.3ex}
one has $a=\km$, $b=h-g$, and $c=-1/g\in (\Z/\km\Z)^*$.
In particular, $b/a=n/\km^2=|I|$ and $\,\ind(\omega_S)=\km$
with $S=\toric(C(I))$.
\end{proposition}

\begin{proof}
Let $a,b,c$ be as in the claim.
By definition, we have $\gcd(a,c)=1$.
We have to show that the generators $\siga=(g,\km)$ and $\sigb=(h,\km)$
of $C(I)$
and the invariants $n:=ab$, $q:=bc-1$ yield isomorphic cones:
First, we clearly obtain that
$|\det(\siga,\sigb)|=(h-g)\km=n$.
It remains to check the characterizing relation
$n|(q\siga+\sigb)$. But this follows from
$$
q\siga+\sigb=\big( (h-g)\,c-1\big)\cdot (g,\km) + (h,\km) \;=\;
\big((h-g)(cg+1),\; (h-g)\,c\km\big)
$$
which is indeed divisible by $n=(h-g)\km$.
Finally, $\,\ind(\omega_S)=a$ by Remark \ref{rem-index}.
\end{proof}

In (\ref{abcNotation}) we mentioned the invariant $c'$ associated
to $(n,q')$ as it was $c$ to $(n,q)$. In the
``interval language'', to switch $q$ and $q'$ means to replace $I$ by $-I$,
i.e.\ to keep $\km$ and to replace $g$ and $h$ by $-h$ and $-g$, respectively.
In particular, this implies that $c'=1/h\in(\Z/\km\Z)^*$.
\\
Moreover, there is a way to visualize both $c$ and $c'$: The
points $[c,\,-\frac{gc+1}{\km}]$ and $[-c',\,\frac{hc'-1}{\km}]$
appear as the first lattice points on the two rays of the shifted cone
$C(I)\dual-[0,\frac{1}{\km}]$, cf.~the proof of Proposition \ref{prop-finalVW}.

\subsection{Grounded cones and intervals}
\label{groundedCones}
To represent two-dimensional cones $\sigma$ by intervals $I$ 
via Proposition \ref{prop-CentrDeg}, the central degree 
$\ko{R}$ played an important role. This leads to the following
notion:

\begin{definition}
\label{def-grounded}
A two-dimensional, polyhedral cone $\sigma$ 
(or the associated interval $I$,
or the associated singularity $S_{n,q}=S_{abc}=\toric(\sigma)$) 
is called {\em grounded} $:\iff$ the central degree $\ko{R}$
belongs to the Hilbert basis
$E=\{\sigr,r^2,\ldots,\sigs\}$ of $\sigma\dual\cap M$,
i.e.\ $\ko{R}$ is irreducible within this semigroup.
If $\ko{R}=r^\ci$, then $\ci$ is called the central index.
\end{definition}

\begin{proposition}
\label{prop-piercInt}
An interval $I\subseteq\Q$ with uniform denominators is grounded if and only if
it contains an interior integer.
\end{proposition}

\begin{proof}
$(\neht)$ We may assume that $0\in\innt I$, i.e.\
$I=[\frac{g}{\km},\frac{h}{\km}]$ with $\km>0$, $g<0$, and $h>0$. 
Then, the dual cone of $\sigma=C(I)$ equals
$\sigma\dual=\langle [-\km,h],\,[\km,-g]\rangle\subseteq 
(\Q\times \Q_{>0})\cup\{[0,0]\}$.
On the other hand, the central degree $\ko{R}$
coincides with
$[0,1]$, and it is obvious that this is irreducible
even within the semigroup $(\Z\times \Z_{>0})\cup\{[0,0]\}$.
\\[0.5ex]
($\then$) 
Let $I=[\frac{g}{\km},\frac{h}{\km}]$ with $\km>0$ and $g<h$,
i.e.\ $\sigma\dual=\langle [-\km,h],\,[\km,-g]\rangle$. We are going to show
that we can obtain $g<0$ and $h>0$ by an integral shift of $I$.
Obviously, we can assume that $0<h<\km$ implying that
$$
[-1,1]\in\innt \langle [-\km,h],\,[0,1]\rangle \subseteq\innt\sigma\dual.
$$
On the other hand, if we had $g>0$, then this would similarly imply that
$$
[1,0]\in\innt \langle [0,1],\,[\km,-g]\rangle \subseteq\innt\sigma\dual.
$$
Then $\ko{R}=[0,1]=[-1,1] + [1,0]$ would be a decomposition 
within $\innt\sigma\dual$.
\end{proof}

Grounded intervals can always be shifted by integers to look like
$I=[-\Ia,\Ib]$ with $\Ia,\Ib\in\Q_{>0}$ (sharing the same denominator).
Then, if $\ci$ denotes the central index,
we can directly express the invariants $\eta_\ci$ and $a_\ci$ from 
(\ref{equSnq})
in terms of $I$:

\begin{proposition}
\label{prop-etaci}
Let $I=[-\Ia,\Ib]$ with $\Ia,\Ib\in\Q_{>0}$ be a 
{\rm(}grounded{\rm)} interval
with uniform denominators. 
Then 
\vspace{-0.3ex}
$$\eta_\ci=1+\min\{\rounddown{\Ia}+\Ib,\, \Ia+\rounddown{\Ib}\},
\hspace{0.8em}
a_\ci=2+\rounddown{\Ia}+\rounddown{\Ib}, 
\hspace{0.6em} \mbox{and} \hspace{0.5em}
|I|=\Ia+\Ib.
\vspace{0.3ex}
$$
\end{proposition}

\begin{proof}
With $\Ia=\frac{-g}{\km}$ and $\Ib=\frac{h}{\km}$, we have
$\sigma\dual=\langle [-\km,h],\,[\km,-g]\rangle$ as usual.
Now, since $\,r^\ci=\ko{R}=[0,1]$ and 
$\{r^{\ci-1},r^\ci\}$ forms a basis of $\Z^2$,
we know that $r^{\ci-1}=[1,\kbb]$, and it has to be the lowest lattice point
above the ray $\Q_{\geq 0}\cdot [\km,-g]=\Q_{\geq 0}\cdot [1,\Ia]$.
Thus, $r^{\ci-1}=[1,\rounddown{\Ia}+1]$ and, similarly,
$r^{\ci+1}=[-1,\rounddown{\Ib}+1]$.
Now, the claim for 
$\eta_\ci=\min\{\frac{\langle\siga,r^{\ci+1}\rangle}{\langle\siga,r^\ci\rangle},
\frac{\langle\sigb,r^{\ci-1}\rangle}{\langle\sigb,r^\ci\rangle}\}$
follows from $\siga=\km\cdot(-\Ia,1)$ and $\sigb = \km\cdot (\Ib,1)$.
\end{proof}

\section{The dualizing sheaf on infinitesimal deformations of $S$}
\label{canX}

The (isomorphism classes of)
infinitesimal $\kk[\keps]$-deformations (with $\keps^2=0$)
of a $\kk$-algebra are gathered in a vector space called $\TT$,
see \cite{janS} for a detailed introduction to deformation theory.
In case of 
toric varieties such as $S=\toric(\sigma)$ from (\ref{toricSnq}), 
the torus $\kT:=\Spec \kk[M]$ acts
on the variety, on the functions, and on all naturally defined modules.
In particular, the vector space $\TT$ becomes $M$-graded.
This can be made explicit by comparing the $M$-degrees of the defining
equations $f$ with those of the perturbation $g$ arising in
$f+\keps g$, cf.~(\ref{constX}).
Thus, the distribution along the degrees of $M$ becomes
the essential information.
We will study the dualizing sheaf $\omega_X$ on the total spaces $X=X_\xi$
for homogeneous elements $\xi\in \TT(S)$.

\subsection{Degrees carrying $\TT$}
\label{degT1}\label{SchlessT1}
Let $\sigma$ be a two-dimensional cone -- we will adopt the notation of
(\ref{toricSnq}) and (\ref{equSnq}).
The dimensions of the homogeneous components 
$\TT(S,-R)$ ($R\in M$) of the finite-dimensional
vector space $\TT(S)$ 
(abbreviated as $\TT(-R)\subseteq\TT$)
are, \cite{pinkham_CQS}:
\begin{enumerate}
\item[(i)]
\label{eckdegs}
$R=r^2$ or $R=r^{e-1}$: $\;\dim_\kk \TT(-R)=1$,
\item[(ii)]
\label{maindegs}
$R=r^i$ for $i=3,\ldots,e-2$: $\;\dim_\kk \TT(-R)=2$,
and
\item[(iii)]
\label{multipledegs}
$R=k\cdot r^i$ for $i=2,\ldots,e-1$ with $2\leq k\leq a_i-1$: 
$\;\dim_\kk \TT(-R)=1$.
\end{enumerate}

We would like to recall Pinkham's method to obtain this 
-- this approach will also provide the major tool for our own calculations of
$\omega_X$. However, unlike the original reference, we will 
consequently use the toric language. It leads to a slightly more
structured description than just naming the dimensions.

\subsubsection{Puncturing}
\label{puncturing}
The main point is to consider deformations of 
the smooth, but non-affine $S\setminus 0$ first. 
They are always locally trivial, and some of them lift to deformations
of $S$. The exact statement for the $\kk[\keps]$- level
is encoded in the exact sequence
$$
0 \to \TT \to \gH^1(S\setminus 0,\theta_S) \to
\gH^1(S\setminus 0, \CO_S^e)
$$
where the latter map is given by $\sum_{i=1}^e dx^{r^i}$. For the upcoming
calculations it is helpful to use this sequence for
redoing the calculation of $\TT(-R)$ for $R\in M$. 
Moreover, since we have a very nice open affine covering
$S\setminus 0=\toric(\siga)\cup\toric(\sigb)$
where we identify $\siga$ and $\sigb$ with the rays they are
generating,
hence $\toric(\siga)$ and $\toric(\sigb)$ are defined
similarly to $\toric(\sigma)$ in (\ref{toricSnq}). Since 
$$
\toric(\siga)\cap\toric(\sigb)=\toric(\siga\cap\sigb)=
\toric(0)=\Spec\kk[M]=\kT,
$$
this is easily done by using \v{C}ech cohomology:

\subsubsection{$\gH^1(S\setminus 0, \CO_S)$}
\label{cechO}
The $1$-\v{C}ech cocycles are 
$\kG(\kT,\CO_S|_\kT)=\kk[M]$, and the $1$-\v{C}ech coboundaries are
generated by the monomials $x^{-R}\in\kk[M]$ with 
$\langle \siga,R\rangle\leq 0$ or $\langle \sigb,R\rangle\leq 0$.
That is, 
$$
\gH^1(S\setminus 0, \CO_S)(-R)=\kk \iff R\in\innt\sigma\dual
\mbox{\ (and $=0$ otherwise).}
$$

\subsubsection{$\gH^1(S\setminus 0, \theta_S)$}
\label{cechTheta}
Here we use the derivations 
$\partial_a\in\theta(-\log\partial S)\subseteq\theta_S$
(with $\partial S:=S\setminus \kT$)
defined by $x^r\mapsto \langle a,r\rangle\cdot x^r$. Via 
$x^s\otimes a\mapsto x^s\partial_a$
they provide an isomorphism 
$\CO_{S}\otimes_\Z N\stackrel{\sim}{\to} \theta(-\log\partial S)$
and, restricted to $S\setminus 0$, the latter sheaf equals $\theta_S$.
\\[0ex]
In particular, the $1$-\v{C}ech cocycles of $\theta_S$ on $S\setminus 0$
are $\kk[M]\otimes N$, and 
the $1$-\v{C}ech coboundaries (in degree $-R$) are generated by 
$x^{-R}\partial_a$ with 
$\langle \siga,R\rangle\leq 1$ or $\langle \sigb,R\rangle\leq 1$.
Thus,
$$
\gH^1(S\setminus 0, \theta_S)(-R)=N_\kk 
\mbox{ for } 
\langle \siga,R\rangle, \langle \sigb,R\rangle\geq 2,
$$ 
and the remaining
cases where $\gH^1(S\setminus 0, \theta_S)(-R)$ is non-vanishing
are 
$\langle \siga,R\rangle=1$, $\langle \sigb,R\rangle\geq 2$ (then it becomes
$N_\kk/\kk\cdot\siga$) and
$\langle \siga,R\rangle\geq 2$, $\langle \sigb,R\rangle=1$ (yielding 
$N_\kk/\kk\cdot\sigb$). 
Note that these cases include
$R=r^2$ and $R=r^{e-1}$, respectively.

\subsubsection{The kernel}
\label{cechKer}
The $i$-th summand $dx^{r^i}$ maps $x^{-R}\partial_a$ to $\langle
a,r^i\rangle\cdot x^{-R+r^i}$. In particular,
whenever $R-r^i\in\innt\sigma\dual$, then $dx^{r^i}$ imposes 
the codimension one condition $\langle a,r^i\rangle=0$
on the derivation $x^{-R}\partial_a$.
\\[1.0ex]
{\em Case 1}. Assume that 
$\langle \siga,R\rangle, \langle \sigb,R\rangle\geq 2$.
Each occurrence of at least two conditions
$R-r^i,R-r^j\in\innt\sigma\dual$ enforces $a=0$.
Using the numbering of the beginning of (\ref{degT1}), the 
remaining cases are
\\[0.5ex]
(ii) where this does not happen at all
yielding $\TT(-R)=N_\kk$, and
\\[0.5ex]
(iii) where $(k\cdot r^i)-r^i=(k-1)\cdot r^i\in\innt\sigma\dual$
leads to the single condition $\langle a, r^i\rangle =0$. 
There we are left with a one-dimensional
$\TT(-R)=(r^i)^\bot\subset N_\kk$.
\\[1.0ex]
{\em Case 2}. Assume that
$\langle \siga,R\rangle=1$ and $\langle \sigb,R\rangle\geq 2$.
Then, 
either 
$\langle \sigb,R\rangle > \langle\sigb,r^1\rangle=n$, i.e.\
$R-r^1\in\innt\sigma\dual$ 
implying the condition
$\langle a, r^1\rangle =0$ forcing $a\in N/\siga\Z$ to become $0$,
or, using the numbering of (\ref{degT1}) again,
\\[0.5ex]
(i) $R=r^2$ with $\TT(-R)=N_\kk/\kk\cdot\siga$.
\\[0.5ex]
The case $\langle \siga,R\rangle\geq 2$, $\langle \sigb,R\rangle=1$
(yielding $R=r^{e-1}$) works similar.

\subsection{The construction of $X_\xi\setminus 0$}
\label{constX}
Let $\xi\in\gH^1(S\setminus 0,\theta_S)$ be given by the
$1$-\v{C}ech cocycle 
$\xi=x^{-R}\partial_a\in \kG(\kT,\theta_S|_\kT)=\kk[M]\otimes N$,
cf.\ (\ref{cechTheta}). The associated infinitesimal
deformation $X_\xi\setminus 0$ of $S\setminus 0$ arises from glueing
the trivial pieces $\toric(\siga)\otimes\kk[\keps]$ 
and $\toric(\sigb)\otimes\kk[\keps]$ along
the $\kk[\keps]$-algebra map
$$
\xymatrix@R=2.6ex@C=3.5em{
\kk[{\siga}\dual\cap M,\;\keps] \ar@{_(->}[d]
&
\kk[{\sigb}\dual\cap M,\;\keps] \ar@{^(->}[d]
\\
\kk[M,\keps] \ar[r]^-{\varphi_\xi}
&
\kk[M,\keps]
}
\vspace{0.5ex}
$$
with $\,\varphi_\xi(x^r):=x^r+\keps\cdot\xi(x^r)=
x^r + \keps\cdot\langle a,r\rangle\cdot x^{r-R}$.
Note that we have decided to use the notation $X_\xi\setminus 0$
even in the case when there is no extension of this to some deformation
$X_\xi$ of the non-punctured $S$.

\subsection{The dualizing sheaf on $X_\xi\setminus 0$}
\label{omegaXPunct}
Let $\xi=x^{-R}\partial_a$ as before.
Since $X_\xi\setminus 0$ is smooth over $\Spec\kk[\keps]$
and since $\omega_{\kk[\keps]}=\kk[\keps]$, it follows from
\cite[p.140]{HartRD} that 
$$
\omega_{X\setminus 0}=\omega_{(X\setminus 0)|\kk[\keps]}
= \Lambda^2\Omega^1_{(X\setminus 0)|\kk[\keps]}.
$$
Choosing a $\Z$-basis $\{\kA,\kB\}$ of $M$,
the local pieces of the latter equal
$$
\textstyle
\omega_{\siga}= \oplus_{\langle \siga,r\rangle\geq 1}\,\kk[\keps]\cdot x^r
\cdot\frac{dx^\kA}{x^\kA}\wedge\frac{dx^\kB}{x^\kB}
\cong
\oplus_{\langle \siga,r\rangle\geq 1}\,\kk[\keps]\cdot x^r\subseteq
\kk[M,\keps]
$$
and similarly for $\omega_\sigb$,
cf.\ \cite[Prop.~8.2.9]{CoxLittleSchenck}.
Note that the isomorphism does, up to sign,
not depend on the choice of $\{\kA,\kB\}$. Now, we determine the
impact of the $\kk[\keps]$-algebra isomorphism
$\varphi_\xi^{[0]}:=\varphi_\xi$ on the glueing $\varphi_\xi^{[1]}$
of the modules 
$\omega_{\siga}|_\kT$ and $\omega_{\sigb}|_\kT$.
Since
$$
\textstyle
\frac{dx^\kA}{x^\kA}\mapsto \frac{d(x^\kA+\keps\langle a,\kA\rangle x^{\kA-R})}
{(x^\kA+\keps\langle a,\kA\rangle x^{\kA-R}} 
=
\frac{dx^\kA}{x^\kA} + \keps \langle a,\kA\rangle\, dx^{-R},
$$
we obtain that $x^r\cdot \frac{dx^\kA}{x^\kA}\wedge\frac{dx^\kB}{x^\kB}$
maps to
$$
\textstyle
\big(x^r\cdot \frac{dx^\kA}{x^\kA}\wedge\frac{dx^\kB}{x^\kB}\big) +
\keps \,x^{r-R}\big(
\langle a,r\rangle \frac{dx^\kA}{x^\kA}\wedge\frac{dx^\kB}{x^\kB}
+ \langle a,\kA\rangle \frac{dx^{-R}}{x^{-R}}\wedge\frac{dx^\kB}{x^\kB}
+ \langle a,\kB\rangle \frac{dx^{\kA}}{x^{\kA}}\wedge\frac{dx^{-R}}{x^{-R}}
\big).
$$
Expressing $R$ within the basis $\{\kA,\kB\}$ 
(and suppressing $\frac{dx^\kA}{x^\kA}\wedge\frac{dx^\kB}{x^\kB}$) 
finally yields
$$
\textstyle
\varphi_\xi^{[1]}:\hspace{1em}
\omega_{\siga}|_\kT\ni \; x^r
\;\longmapsto\;
x^r+\keps\cdot\langle a,r-R\rangle\cdot x^{r-R}
\; \in \omega_{\sigb}|_\kT.
$$
This description enables us to determine the
class $[\omega_{X\setminus 0}]\in
\gH^1(S\setminus 0, \CO_X^*)$. If $x^\kra$ and $x^\krb$ are generators of
$\omega_{\siga}$ and $\omega_{\sigb}$, respectively,
i.e.\ if $\langle \siga,\kra\rangle = \langle \sigb,\krb\rangle=1$,
then $[\omega_{X\setminus 0}]$ is represented by the
$1$-\v{C}ech cocycle 
$\varphi(x^\kra)/x^\krb\in\kG(\kT,\CO_{\kT\otimes\kk[\keps]}^*)$.
It is equal to
$$
\psi_1:=x^{\kra-\krb} + \keps\cdot\langle a,\kra-R\rangle\cdot x^{\kra-\krb-R}.
$$
Similarly, we might consider the 
glueing map $\varphi_\xi^{[\kr]}$ for a reflexive power 
$\omega_{X\setminus 0}^{[\kr]}$
instead of just for $\omega_{X\setminus 0}$. 
Then, the previous calculations yield

\begin{lemma}
\label{lem-cocycle}
$
\textstyle
\,\varphi_\xi^{[\kr]}:
x^r
\mapsto
x^r+\keps\cdot\langle a,r-\kr R\rangle\cdot x^{r- R},
$
\vspace{-0.5ex}
and the $1$-\v{C}ech cocycle becomes
$$
\psi_\kr:= x^{\kra(\kr)-\krb(\kr)} + 
\keps\cdot\langle a,\kra(\kr)-\kr R\rangle\cdot 
x^{\kra(\kr)-\krb(\kr)-R} \in \kG(\kT,\CO_{\kT\otimes\kk[\keps]}^*)
\vspace{-0.5ex}
$$
with $\kra(\kr),\krb(\kr)\in M$ satisfying
$\langle \siga,\kra(\kr)\rangle = \langle \sigb,\krb(\kr)\rangle=\kr$.
\end{lemma}

Note that one might take, 
if some $\kra=\kra(1)$ and $\krb=\krb(1)$ are available,
the multiples $\kra(\kr)=\kr\cdot\kra$ and
$\krb(\kr)=\kr\cdot\krb$ for a general $\kr\in\Z$.
However, in (\ref{emailVersion}) we will prefer a 
different choice for $\kr=\km$.

\subsection{Extending functions along codimension two}
\label{extCodim2}
Let $S=\toric(\sigma)$ be as before.
Since it is normal, it carries the Hartogs property $S_2$ 
as it was asked for in (\ref{QG.V.defn}).
Now, if $A$ is an Artinian $\kk$-algebra 
and $X$ is a deformation of $S$ over $A$
(we just need the case $A=\kk[\keps]$ here),
we would like to keep this property. 

\begin{lemma}
\label{lem-reflexive}
If $\CF$ is a reflexive $\CO_X$-module, i.e.\
if $\CF=\lHom_{\CO_X}(\CG,\CO_X)$ for some $\CO_X$-module $\CG$, then
$\,\kG(S,\CF)\to\kG(S\setminus 0,\,\CF)$ is an isomorphism.
\end{lemma}

\begin{proof}
It suffices to check this for $\CF=\CO_X$.
We proceed by induction. Choosing a non-trivial element $\keps\in A$
with $\keps\cdot\idm_A=0$, we obtain an exact sequence
of $A$-modules
$$
0 \to \kk \stackrel{\cdot\keps}{\to} A \to \ko{A}\to 0.
$$
Denoting $\ko{X}:=X\otimes_A\ko{A}$, flatness, restriction to $S\setminus 0$,
and taking global sections provides the the following commutative diagram with
exact rows:
$$
\xymatrix@R=2.6ex@C=1.5em{
0 \ar[r] & 
\kG(S,\CO_S) \ar[r]^-{\cdot\keps} \ar[d]^-{\sim}&
\kG(S,\CO_X) \ar[r] \ar[d] &
\kG(S,\CO_{\ko{X}}) \ar[r] \ar[d]^-{\sim}&
0 \ar[d]
\\
0 \ar[r] & 
\kG(S\setminus 0,\,\CO_S) \ar[r]^-{\cdot\keps} &
\kG(S\setminus 0,\,\CO_X) \ar[r] &
\kG(S\setminus 0,\,\CO_{\ko{X}}) \ar[r]^-{0} &
\gH^1(S\setminus 0,\,\CO_S).
}
$$
Now, the claim follows from the 5-lemma.
\end{proof}

\subsection{The dualizing sheaf on $X_\xi$}
\label{omegaX}
In contrast to (\ref{constX}) we now start with a
$$
\xi=x^{-R}\,\partial_a \in
\TT(-R)\subseteq \gH^1(S\setminus 0,\,\theta_S)(-R)
\jrus \kk[M] \otimes_\Z N
$$
yielding a true $X=X_\xi$ and not just a punctured
$X_\xi\setminus 0$.
The extension theorem along two-codimensional subsets
gives us the right tool to understand $\omega_X^{[\kr]}$
out of $\omega_{X\setminus 0}^{[\kr]}$ we have obtained in
(\ref{omegaXPunct}), namely 
$$
\textstyle
\omega^{[\kr]}_{X} = \kG(S\setminus 0,\,\omega^{[\kr]}_{X})
= 
\{f= \sum_{\siga\geq \kr} (c_r+\keps d_r) \,x^r\kst
\varphi_\xi^{[\kr]}(f)\in \oplus_{\sigb\geq \kr} \,\kk[\keps]\cdot x^r\}.
$$
From Lemma \ref{lem-cocycle} we know that
$$
\textstyle
\varphi_\xi^{[\kr]}\big(\sum_{r} (c_r+\keps d_r) x^r\big) =
\sum_{r} c_r\,x^r + \keps\cdot\big(
\sum_r d_r\,x^r + c_r \langle a,r-\kr R\rangle\cdot x^{r-R}\big),
$$
and the combination of these two statements yields

\begin{lemma}
\label{lem-omegaXg}
The sum $f= \sum_{r\in M} (c_r+\keps d_r)\, x^r$ belongs to
$\omega^{[\kr]}_{X}$ if and only if
the following two assertions hold
\begin{enumerate}
\item[(i)]
$\langle\siga,r\rangle<\kr$ implies $c_r=d_r=0$
\item[(ii)]
$\langle\sigb,r\rangle<\kr$ implies $c_r=0$
and $d_r=c_{r+R}\cdot\big\langle a,\,(\kr-1)R-r\big\rangle$.
\end{enumerate}
\end{lemma}

\subsection{Surjectivity of the restriction map
$\CR_\xi^{[\kr]}$}
\label{surjOmega}
Recall from (\ref{QG.V.defn}) that one of the characterizations of the
property $\kiso{\kr}$ for $X=X_\xi$ was the surjectivity of the restriction map
$\CR_\xi^{[\kr]}:\omega_X^{[\kr ]}\to\omega_S^{[\kr]}$.
Since this map 
just sends $\keps\mapsto 0$, i.e.\
$$
\textstyle
\CR_\xi^{[\kr]}:\;
\sum_{r} (c_r+\keps d_r) \,x^r \mapsto \sum_{r} c_r \,x^r
$$
and $\omega_S^{[\kr]}=\{
\sum_{\langle\siga,r\rangle, \langle\sigb,r\rangle \geq \kr} c_r\,x^r\}$
by \cite[Prop.~8.2.9]{CoxLittleSchenck},
we can characterize this property by

\begin{lemma}
\label{lem-surjOmega}
$\xi=x^{-R}\,\partial_a$ satisfies $\kiso{\kr}$ 
$\iff$ each $r\in M$ with both
$$
\kr\leq\langle \siga,r\rangle<\kr+\langle \siga,R\rangle
\hspace{0.8em}\mbox{and}\hspace{0.8em}
\kr\leq\langle \sigb,r\rangle<\kr+\langle \sigb,R\rangle
$$
leads to $\langle a, \kr R-r\rangle =0$.
\end{lemma}

\begin{proof}
The second part of Condition (ii) for $f\in \omega^{[\kr ]}_{X}$
in Lemma \ref{lem-omegaXg}
can be read as
that $\langle\sigb,r-R\rangle<\kr $ implies 
$d_{r-R}=c_{r}\cdot\big\langle a,\,\kr R-r\big\rangle$.
Hence, together with $\langle\siga,r-R\rangle<\kr $, this would
enforce that $c_{r}\cdot\big\langle a,\,\kr R-r\big\rangle=0$.
On the other hand, if 
$\langle\siga,r\rangle, \langle\sigb,r\rangle \geq \kr$,
then $c_r\neq 0$ is allowed in $\omega_S^{[\kr ]}$.
\end{proof}

For given $R\in M$ and $\kr\in\Z$ we define the following zones within $M_\Q$:
$$
Z_{R,\,\kr}:=
\{r\in M_\Q\kst \kr\leq\langle \siga,r\rangle<\kr+\langle \siga,R\rangle
\hspace{0.5em}\mbox{and}\hspace{0.5em}
\kr\leq\langle \sigb,r\rangle<\kr+\langle \sigb,R\rangle\}.
$$
Then, for $\xi=x^{-R}\,\partial_a$, the previous lemma says that 
$$
\kiso{\kr} \;\iff\;
\CR_\xi^{[\kr]} \mbox{ is surjective } \;\iff\;
Z_{R,\,\kr }\cap M\subseteq a^\bot+\kr R.
$$

Actually, up to the point that $X_\xi$ does not make sense otherwise,
we did not use $\xi\in \TT$ so far. This property is
equivalent to $\kiso{0}$, and it will be discussed in (\ref{zonesT1}).

\section{V-deformations}
\label{V-deformations}
Let $\sigma=\langle\siga,\sigb\rangle$ be as before, e.g.\
it can be obtained as the cone $C(I)$ over an
interval with uniform denominators $I=[\frac{g}{\km},\frac{h}{\km}]$ 
as in (\ref{coneInterval}).
\\
In the present section, we will approach the V-deformations 
of $S=\toric(\sigma)$ defined in (\ref{def-V}). Since 
$\ind(\omega_S)=\km$, we will mostly study the property 
$\kiso{\km}$ for a given infinitesimal deformation $\xi=x^{-R}\partial_a$.

\subsection{Shifting the zones}
\label{shiftZones}
Recall from Remark \ref{rem-index} that
$\frac{\sigr+\sigs}{n}=\frac{1}{\km}\cdot\ko{R}\in M_\Q$ 
\vspace{0.5ex}
is the truly canonical (but rational) degree. 
Moreover, depending on $R\in M$ we denote
$$
\textstyle
Z_R := Z_{R,\,0}=
\sigma\dual\cap(R-\innt\sigma\dual) \subseteq M_\Q.
$$
The degrees $R$ we are interested in are always elements of
$\innt\sigma\dual$. In particular, $Z_R$ is then a bounded region --
it is a half-open parallelogram 
having $0$ and $R$ as opposite vertices. While these two vertices belong to
the lattice $M$, the remaining ones usually do not. The relation to
the zones $Z_{R,\,\kr }$ from (\ref{surjOmega}) is
$$
\textstyle
Z_{R,\,\kr }= \frac{\kr}{\km}\,\ko{R} +Z_R 
\hspace{1.5em}\mbox{for every}\hspace{0.5em}
\kr\in\Z.
$$
In particular, $Z_{R,\,\kr + \km}= \ko{R} + Z_{R,\,\kr }$, 
i.e.\ the zones $Z_{R,\,\kr +\Z \km}$ just differ by integral translation. 
This gives rise to define the ``stable'' condition
$$
\textstyle
\kISO{\kr}  := \bigcap_{\ell\in\Z} \hspace{0.2em} \kiso{{\kr +\ell \km}}
\hspace{1.5em}\mbox{(still being a condition for $\xi=x^{-R}\partial_a$).}
$$
That is, the condition that $\kiso{\kr}$ is true for all $\kr\in\Z$ can be
replaced by the finite one asking for $\kISO{\kr}$ for all $\kr\in\Z/\km\Z$.
Moreover, to be a V-deformation in the sense of (\ref{def-V}) means
to fulfill the condition $\kISO{0}$.

\begin{proposition}
\label{prop-shift}
{\rm 1)} If $Z_{R,\,\kr}\cap M =\emptyset$, then $\kISO{\kr}$ is fulfilled
for each $a\in N$.
\\[0.5ex]
{\rm 2)} Assume that $Z_{R,\,\kr }\cap M\neq\emptyset$.
Then the following conditions are equivalent:
\vspace{-0.5ex}
$$
\kISO{\kr}
\;\iff\; \kiso{{\kr +\ell \km}} \mbox{\rm\ for two different } \ell\in\Z
\;\iff\; \kiso{\kr } \mbox{\rm\ and } a\in(\ko{R}-\km R)^\bot.
$$
\end{proposition}

\begin{proof}
Let $\ell\in\Z$. If $r\in Z_{R,\,\kr}\cap M$, then
$r+\ell\, \ko{R}\in Z_{R,\,\kr +\ell \km}$. Hence, the conditions
$\kiso{\kr}$ and $\kiso{{\kr +\ell \km}}$ mean that
$$
\langle a,\,\kr R-r\rangle =0 
\hspace{1em}\mbox{and}\hspace{1em}
\langle a,\,(\kr +\ell \km)R-(r+\ell \ko{R})\rangle =0,
$$
respectively.
However, the difference of the two left hand sides equals
$\langle a, \,-\ell \km R +\ell \ko{R}\rangle = 
\ell\cdot\langle a,\,\ko{R}-\km R\rangle$.
\end{proof}

\begin{corollary}
\label{cor-shift}
{\rm 1)} 
If $\,\xi=x^{-R}\partial_a$ is a V-deformation, then
$a\in(\ko{R}-\km R)^\bot$.
{\rm (See the upcoming Corollary \ref{cor-nqG} for a stronger statement.)}
\\[0.5ex]
{\rm 2)} Assume that $a\in(\ko{R}-\km R)^\bot\setminus\{0\}$. 
Then, for any $\kr\in\Z$, $\kiso{\kr}$ {\rm (}or even $\kISO{\kr} ${\rm)} 
is equivalent to $(Z_{R,\,\kr}\cap M)-\kr R\subseteq\Q\cdot(\ko{R}-\km R)$.
\vspace{0.5ex}
Likewise, this condition is equivalent to
$(Z_{R,\,\kr}\cap M)-\frac{\kr}{\km}\,\ko{R}\subseteq\Q\cdot(\ko{R}-\km R)$.
\end{corollary}

\begin{proof}
(1) follows from the fact that there are non-empty $(Z_{R,\,\ell\km}\cap M)$
whenever $R\in\innt\sigma\dual$ (and only those $R$ matter for
$\TT(-R)\neq 0$): Just take $\ell=0$.
\\[0.5ex]
(2) $\kiso{\kr}$ means that for each $r\in Z_{R,\,\kr }\cap M$ we have
$a\in (r-\kr R)^\bot$. Together with $a\in(\ko{R}-\km R)^\bot$ this means that
$a$ can be non-trivial if and only if
both $r-\kr R$ and $\ko{R}-\km R$ are collinear.
Moreover, $\ko{R}-\km R$ does never vanish (since $\toric(\sigma)\neq A_k$).
\end{proof}

\subsection{Focusing on $\TT$-degrees}
\label{zonesT1}
For investigating the $\kiso{\kr}$ property we did not use yet that
the set of degrees $R\in M$ with $\TT(-R)\neq 0$ is very restricted.
Taking this into account implies

\begin{lemma}
\label{lem-deg0}
Every deformation $x^{-R}\partial_a\in \TT(-R)$ satisfies $\kiso{0}$.
\end{lemma}

\begin{proof}
Actually, this statement is trivial -- the condition 
$\kiso{0}$ means that $\omega_X^{[0]}=\CO_X$ is flat over $k[\keps]$, 
i.e.\ it even
characterizes the elements of $\TT(-R)$. Nevertheless, 
e.g.\ to practice our new language involving the zones $Z_R$,
we would like to present a direct argument, too:
\\[1ex]
Condition $\kiso{0}$ means $Z_R\cap M\subseteq a^\bot$
with $Z_R=\sigma\dual\cap(R-\sigma\dual)$.
According to (\ref{degT1}), we distinguish between two cases:
\\[0.5ex]
(i)+(ii) $\,R=r^i$ with $i=2,\ldots,e-1$:
Since these elements are irreducible in the semigroup $\sigma\dual\cap M$,
we obtain $Z_R\cap M=\{0\}$, and this belongs to every $a^\bot$.
\\[0.5ex]
(iii) $\,R=k\cdot r^i$ for $i=2,\ldots,e-1$ with $2\leq k\leq a_i-1$:
Here we have
$$
Z_R\cap M=\{0,r^i,\ldots,(k-1)r^i\},
$$
i.e.\ Condition $\kiso{0}$ means $\langle a,r^i\rangle =0$.
However, by (\ref{cechKer}), Case 1,
exactly this is ensured to hold true within
$\TT(-k\cdot r^i)$.
\end{proof}

{\em Remark.}
Actually, the condition 
$$
\kiso{0} \;\iff\; 
Z_R\cap M\subseteq a^\bot \;\iff\; 
a\in (Z_R\cap M)^\bot
$$ 
should be understood as an alternative description of $\TT(-R)$.
However, this is not new -- it coincides with the description in
\cite[(2.2)]{flip}. There, one has defined the finite subsets
$$
E_\siga^R:=\{r\in E\kst \langle\siga, r\rangle < \langle\siga, R\rangle\}
\hspace{1em}\mbox{and}\hspace{1em}
E_\sigb^R:=\{r\in E\kst \langle\sigb, r\rangle < \langle\sigb, R\rangle\}
$$
of $M$, and this lead to an exact sequence
$$
0 \to \TT(-R) \to\big(\spann_k E_\siga^R\cap \spann_k E_\sigb^R\big)^*
\to \spann_k(E_\siga^R\cap E_\sigb^R)^* \to 0.
$$
In particular, $\TT(-R)$ is a subquotient of $N_k$, and 
$a\in \TT(-R)$ if and only if $a\in(E_\siga^R\cap E_\sigb^R)^\bot$.
Now, the relation to our condition $\kiso{0}$ is
that $E_\siga^R\cap E_\sigb^R=Z_R\cap E$.

\begin{corollary}
\label{cor-nqG}
$\xi=x^{-R}\partial_a\in \TT(-R)$ is a V-deformation,
i.e.\ it fulfills the stable condition $\kISO{0}$, if and only if
$a\in (\ko{R}-\km R)^\bot$.
\end{corollary}

\begin{proof}
This follows from Proposition~\ref{prop-shift}\,(2). The implication
($\then$) was already stated in Corollary \ref{cor-shift}.
The reversed implication ($\neht$) makes use of Lemma \ref{lem-deg0}.
\end{proof}

\subsection{Counting V-deformations}
\label{restrictPhi}
We run through the list (i)-(iii) of (\ref{degT1}) and especially
(\ref{cechKer})
to determine $(\ko{R}-\km R)^\bot=(\sigr+\sigs-nR)^\bot$, 
i.e.\ the V-deformations within each homogeneous summand $\TT(-R)$.
\\[1ex]
(i) $R=r^2$ (and similarly $R=r^{e-1}$):
$\,\TT(-r^2)=N_\kk/\kk\cdot\siga$. 
The element $\sigr+\sigs -n r^2$ 
is contained in ${\siga}^\bot\subset M$,
hence it provides a linear map $N/\Z\cdot\siga\to\Z$ where
$\TTV(-r^2)$ is generated by the kernel. Since we had excluded the
$A_{n-1}$-singularity, this linear map is also non-trivial,
i.e.\ there is no V-deformations in degree $-r^2$ (and $-r^{e-1}$).
\\[1ex]
(ii)
$R=r^i$ for $i=3,\ldots,e-2$: $\; \TT(-r^i)=N_\kk$.
We know that $\sigr+\sigs -n r^i$ is again non-trivial,
hence it provides a one-dimensional kernel within the two-dimensional
$\TT(-r^i)$.
Altogether, this yields an $(e-4)$-dimensional space of V-deformations.
\\[1ex]
(iii)
$R=k\cdot r^i$ for $i=2,\ldots,e-1$ with $2\leq k\leq a_i-1$:
$\; \TT(-kr^i)=(r^i)^\bot\subset N_\kk$.
Here we obtain $\TTV(-kr^i)=(\sigr+\sigs,\, r^i)^\bot$, i.e.\
this is non-trivial if and only if
$\sigr+\sigs\in\N\cdot r^i$, i.e.\ if $\sigma$ is grounded
(see Definition \ref{def-grounded}) and $r^i=\ko{R}$ is the central degree
(i.e.\ $i=\ci$ is the central index).
If this is the case, and if $\sigma$ stems from
an interval $I=[-\Ia,\Ib]$ as in Proposition \ref{prop-etaci},
then we gather another $a_\ci-2=\rounddown{\Ia}+\rounddown{\Ib}\,$ 
\mbox{V-deformations}.

\subsection{Representing $\TTV$ as a kernel}
\label{emailVersion}
An alternative approach to visualize the V-deformations of
$S=\toric(\sigma)$ is to consider the following
map $\Phi:\TT\to \gH^2_0(S,\CO_S)$:
If $\xi\in \TT$ is represented by an infinitesimal deformation 
$X=X_\xi\to\Spec\kk[\keps]$, then there is an exact sequence
$$
\xymatrix@R=0.4ex@C=1.5em{
0 \ar[r] & 
\CO_S \ar[r]^{} & 
\CO_X^* \ar[r] &
\CO_S^* \ar[r] &
1
\\
& f \ar@{|->}[r] &
1+\keps f.
}
$$
Since the map 
$\gH^0(S\setminus 0,\CO_X^*)\to \gH^0(S\setminus 0,\CO_S^*)$
equals $\CO_X^*\surj\CO_S^*$ on the affine $S$, i.e.\ it is 
notably surjective, this implies the exactness of
$$
0 \to \gH^1(S\setminus 0,\CO_S) \to \gH^1(S\setminus 0,\CO_X^*)
\to \gH^1(S\setminus 0,\CO_S^*).
$$
Thus, using that $\omega_S^{[\km]}$ is trivial,
we may define $\Phi(\xi)$ as the class of $\omega_X^{[\km]}$ in 
$$
\xymatrix@R=0.6ex@C=1.5em{
\gH^2_0(S,\CO_S)=\gH^1(S\setminus 0,\CO_S) = \ker\Big(
\hspace{-2em}
&
\gH^1(S\setminus 0,\CO_X^*) \ar[r] &  \gH^1(S\setminus 0,\CO_S^*)& 
\hspace{-3.5em}\Big)
\\
&
\Pic(X_\xi\setminus 0)\ar@{=}[u] \ar[r] 
&  \Pic(S\setminus 0). \ar@{=}[u]
}
$$
Comparing with the flatness part of the
definition at the end of (\ref{QG.V.defn}), it follows that
the kernel $\ker\Phi\subseteq \TT$ consists
exactly of the deformations satisfying $\kiso{\km}$, i.e., of
the V-deformations, cf.~(\ref{def-V}).
\\[1ex]
It turns out that $\Phi$ can be extended to 
$\gH^1(S\setminus 0,\theta_S)$, i.e.\
we consider (locally trivial) deformations
$X_\xi\setminus 0$ of (the smooth) $S\setminus 0$ again.
Using the descriptions
of $\gH^1(S\setminus 0,\theta_S)$ and
$\gH^1(S\setminus 0,\CO_S)$ given in \ref{cechTheta} and \ref{cechO},
respectively, the final result fits perfectly with
Corollary \ref{cor-nqG}:

\begin{proposition}
\label{prop-calcPhi}
Let $R\in\innt\sigma\dual\cap M$. 
Then, the degree $-R$ part of $\Phi$ 
$$
\Phi(-R):\gH^1(S\setminus 0,\theta_S)(-R)\to
\gH^1(S\setminus 0,\CO_S)(-R)
$$
is given by 
$\,x^{-R}\partial_a\mapsto \langle a,\ko{R}-\km R\rangle\cdot x^{-R}$.
In other words, using the natural maps
$N_\kk\surj \gH^1(S\setminus 0,\theta_S)(-R)$ and
$\gH^1(S\setminus 0,\CO_S)(-R)=\kk$, the map 
$\Phi(-R):N_k\to k$ equals $\ko{R}-\km R\in M$.
\end{proposition}

\begin{proof}
In (\ref{omegaXPunct}) we have dealt with 1-cocycles of $\CO_X^*$,
and in Lemma \ref{lem-cocycle} we have obtained 
an element $\psi_\km$ describing the class of $\omega_X^{[\km]}$
after using the surjection
$\Gamma(\kT,\CO_{\kT\otimes k[\keps]}^*)\surj \gH^1(S\setminus 0,\CO_X^*)$.
Restricting $\psi_\km$ via $\keps\mapsto 0$ to
$\Gamma(\kT,\CO_\kT^*)\surj \gH^1(S\setminus 0,\CO_S^*)$
yields $x^{\kra(\km)-\krb(\km)}$.
\\[0.5ex]
Since $\omega_S^{[\km]}=\CO_S$, this is a $1$-\v{C}ech coboundary.
One can see this directly by the possibility of choosing
$\kra(\km)=\krb(\km)=\ko{R}$ 
-- then $x^{\kra(\km)-\krb(\km)}$ becomes $1$ right away.
Applying this recipe to the original cocycle $\psi_\km$ 
of Lemma \ref{lem-cocycle} as well, we obtain that
$$
\psi_\km = 1 + \keps\cdot\langle a,\,\ko{R}-\km R\rangle\cdot x^{-R}. 
$$
Recall that the second map within the exact sequence
$$
0 \to \CO_{S\setminus 0} \to
\CO_{X\setminus 0}^* \to
\CO_{S\setminus 0}^* \to 1
$$
sends $f\mapsto 1+\keps\cdot f$.
Thus, $\Phi(\xi)=[\omega_{X\setminus 0}^{[\km ]}]\in 
\gH^1(S\setminus 0,\CO_S)(-R)$
is given by the $1$-\v{C}ech cocycle
$\langle a,\,\ko{R}-\km R\rangle\cdot x^{-R}\in \kk[M]=\kG(\kT,\CO_S)$.
\end{proof}

\section{\qG- and \VW-deformations}
\label{compJanos}
V-deformations are understood,
by Corollary \ref{cor-nqG} and (\ref{restrictPhi}).
Next we will turn to the stronger \qG- and 
\VW-deformations,
cf.\ (\ref{def-qG}) and (\ref{def-VW}) for the definition
of these notions.

\subsection{Extending the lattice $M$}
\label{extendM}
We adopt the notation from (\ref{coneInterval}).
For $I=[\frac{g}{\km},\frac{h}{\km}]$ we know that 
$\frac{1}{\km}\ko{R}\in M_\Q$ is the truly canonical, but rational
degree. This gives rise to an enlargement of our
lattice $M$, namely
$$
\textstyle
\til{M}:=M+\Z\cdot \frac{1}{\km}\ko{R}.
$$
Let us assume that $\xi=x^{-R}\partial_a\in \TT(-R)$ is a V-deformation,
i.e.\ $\langle a,\,\ko{R}-\km R\rangle=0$.
Using the zone $Z_R=\sigma\dual\cap(R-\innt\sigma\dual)$
defined in (\ref{shiftZones}), we obtain

\begin{proposition}
\label{prop-WqG}
{\rm 1)} $\xi$ is a \qG-deformation $\iff$ 
$\til{M}\cap Z_R\subseteq \Q\cdot (\ko{R}-\km R)$, and
\\[0.5ex]
{\rm 2)} $\xi$ is a \VW-deformation $\iff$
$(M+\frac{1}{\km}\ko{R})\cap Z_R\subseteq \Q\cdot (\ko{R}-\km R)$.
\end{proposition}

Note that the difference between both cases just arises from the 
tiny difference
between $\til{M}=M+\Z\cdot \frac{1}{\km}\ko{R}$ and $M+\frac{1}{\km}\ko{R}$.

\begin{proof}
Recall from (\ref{shiftZones}) that
$Z_{R,\,\kr }= \frac{\kr}{\km}\,\ko{R} +Z_R$ for $\kr\in\Z$.
Thus, Corollary \ref{cor-shift}\,(2) says that 
the conditions $\kiso{\kr}$ and $\kISO{\kr}$ are equivalent to
$$
\textstyle
Z_R\cap (M-\frac{\kr}{\km}\ko{R})=
((Z_{R} + \frac{\kr}{\km}\ko{R})\cap M)- \frac{\kr}{\km}\ko{R}
\subseteq\Q\cdot(\ko{R}-\km R).
$$
While $\kr=-1$ directly leads to (2),
one uses $\bigcup_{\kr\in\Z}(M-\frac{\kr}{\km}\ko{R})=\til{M}$ for (1).
\end{proof}

Now, we are going to scan the degrees of $\TTV$ 
listed in (\ref{restrictPhi})(ii) and (iii) for \qG- and \VW-deformations.
(Note that the deformations in (\ref{restrictPhi})(i) 
are not even V-deformations.)
\\
Actually, it is convenient to proceed with a minor change to the division into
the two cases:
We will shift (and this applies only to the grounded case)
the central degree $\ko{R}=r^\ci$ from Class (ii) to (iii).
Thus, in (ii) we now collect exactly the non-central $R=r^i$ ($i=3,\ldots,e-2$),
and Class (iii) will
gather all $R=k\cdot r^\ci$ with $1\leq k\leq a_\ci-1$.
Note that this set is empty unless $\sigma$ is grounded, i.e.\ $r^\ci=\ko{R}$.

\subsection{The degrees of (\ref{restrictPhi})(ii)}
\label{scanII}
Let $R=r^i$ with $i=3,\ldots,e-2$ be a non-central degree.
The latter property can be expressed by
$$
\textstyle
\frac{1}{\km}\ko{R}\;\notin\;\Q\cdot (\ko{R}-\km r^i).
$$
On the other hand, we know that
$$
\textstyle
\langle\siga,\frac{1}{\km}\ko{R}\rangle=
\langle\sigb,\frac{1}{\km}\ko{R}\rangle=1
\hspace{0.7em}\mbox{and}\hspace{0.7em}
\langle\siga,r^i\rangle, \langle\sigb,r^i\rangle >1
$$
which implies that
$\frac{1}{\km}\ko{R}\,\in\, \sigma\dual\cap(r^i-\innt\sigma\dual) = Z_{r^i}$.
Hence,
$$
\textstyle
\frac{1}{\km}\ko{R}\;\in\; (M+\frac{1}{\km}\ko{R})\cap Z_{r^i}.
$$
Applying 
Proposition \ref{prop-WqG}\,(2),
this shows that
the deformations of degree $r^i$ cannot be \VW-deformations,
let alone \qG-deformations. In other words,
the property of being a grounded singularity is a necessary condition
for the existence of \VW- or \qG-deformations.

\subsection{The degrees of (\ref{restrictPhi})(iii)}
\label{scanIII}
Let $\sigma$ be a grounded cone with central degree $\ko{R}=r^\ci$.
From (\ref{coneInterval}) and (\ref{groundedCones}) 
we know that $\sigma=\langle\siga,\sigb\rangle$ can be obtained as 
$C(I)=\langle (g,\km),\,(h,\km)\rangle$ 
from a grounded interval $I=[\frac{g}{\km},\frac{h}{\km}]=[-\Ia,\Ib]$
with $\km>0$, $\,g<0<h$, and $\gcd(g,\km)=\gcd(h,\km)=1$. 
In particular, $\Ia,\Ib\in\Q_{>0}$, 
the central degree $\ko{R}\in M$ becomes $[0,1]\in\Z^2$,
and $\sigma\dual=\langle [-\km,h],\,[\km,-g]\rangle$.
\\[0.5ex]
Let $R=k\cdot \ko{R}$ with
$k=1,\ldots,(a_\ci-1)=1+\rounddown{\Ia}+\rounddown{\Ib}$
(cf.\ Proposition~\ref{prop-etaci}).
Note that we have included the case $k=1$ 
originally belonging to (\ref{restrictPhi})(ii).
The zone $Z_{k\ko{R}}$ is the half open parallelogram in $M_\Q=\Q^2$ 
with the vertices 
$$
\textstyle
[0,0],\hspace{1.2em}
\frac{k}{h-g}\cdot [\km,-g], \hspace{1.2em}
[0,k], \hspace{1.2em}
\frac{k}{h-g}\cdot [-\km,h].
$$
Moreover, the line $\Q\cdot(\ko{R}-\km R)=\Q\cdot\ko{R}$
we are interested in by Proposition \ref{prop-WqG} is given
by the diagonal $\ko{[0,0]\,[0,k]}$.

\subsubsection{qG-deformations}
\label{IIIqG}
From (\ref{restrictPhi})(iii) we know that each $k=1,\ldots,(a_\ci-1)$
gives rise to a one-dimensional 
$\TTV(-k\cdot\ko{R})=\ko{R}^\bot\subseteq N_k$. For each
of these $k$ we have to decide whether
$\TTqG(-k\cdot\ko{R})=0$ or $\ko{R}^\bot$.

\begin{proposition}
\label{prop-finalQG}
The \qG-deformations of $S$ consist exactly of the
one-dimensional subspaces $\ko{R}^\bot\subseteq \TT(-k\cdot\ko{R})$
with $1\leq k\leq \min\{a_\ci-1,\,|I|\}$.
\end{proposition}

\begin{proof}
We consider the embedding $\iota:\til{M}\hookrightarrow\Z^2$ 
obtained by evaluating
$(\siga,\sigb)$. Actually, restricting to $M=\Z^2$, this
reflects the original situation of $M=(\Z^2)^G$,
and $\iota|_M$ is given by the matrix
$\Matc{2}{g & \km \\ h & \km}$.
The rational $\iota_\Q$ is an isomorphism, we can detect 
$\iota_\Q(\sigma\dual)=\Q^2_{\geq 0}$,
and the new, truly canonical degree
$\frac{1}{\km}\ko{R}=[0,\frac{1}{\km}]\in\til{M}$ maps to $[1,1]$.
\\[0.5ex]
We are going to apply Proposition \ref{prop-WqG}.
The description by $\iota$ implies that 
$(\til{M}\cap Z_{k\ko{R}})\setminus\Q\ko{R}$ is
non-empty if and only if $Z_{k\ko{R}}$ contains an $\til{M}$-lattice
point on the boundary $\partial\sigma\dual\setminus\{0\}$ --
just subtract $\iota(\frac{1}{\km}\ko{R})=[1,1]$ 
whenever the boundary is not reached yet.
\\[1.0ex]
While $r^1=[\km,-g]$ used to be a primitive generator of 
one ray of $\sigma\dual$ within the lattice $M=\Z^2$, 
this is no longer true in $\til{M}=\Z\times\Z\frac{1}{\km}$.
Here, the element $[1,\frac{-g}{\km}]$ does the job instead. 
Thus, it remains to
check whether this generator belongs to $Z_{k\ko{R}}$.
Since $\langle\siga,[1,\frac{-g}{\km}]\rangle=0$
and $\langle\sigb,[1,\frac{-g}{\km}]\rangle=h-g$,
this leads to the condition $h-g<km$.
The situation for the other ray generated by $r^e=[-\km,h]$ is the same.
\end{proof}

\subsubsection{\VW-deformations}
\label{IIIVW}
Recall from (\ref{normAct}) that $q'\in(\Z/n\Z)^*$ denotes the multiplicative
inverse of $q$. It is, like $q$ itself, assumed to be normalized as
$1\leq q'<n-1$. The singularities $S_{n,q}$ and $S_{n,q'}$ are isomorphic,
by (\ref{abcNotation}) they share $a'=a$ (that is $\km'=\km$) and $b'=b$,
and at the end of (\ref{coneInterval}) we have seen that
$c=-1/g$ and $c'=1/h$ in $(\Z/\km\Z)^*$.
\\[1ex]
As before,  we are in the grounded case, and we consider a
$k\in\{1,\ldots,a_\ci-1\}$
with $a_\ci-1=1+\rounddown{\Ia}+\rounddown{\Ib}$.

\begin{proposition}
\label{prop-finalVW}
$\,\TTVW(-k\cdot\ko{R})\neq 0$
$\;\iff\;$ 
$k\leq \min\{\frac{q+1}{a}, \frac{q'+1}{a}\}=
\min\{c\cdot |I|, \,c'\cdot |I|\}$.
\end{proposition}

\begin{proof}
By Corollary \ref{cor-shift}\,(2) or
Proposition \ref{prop-WqG}, a degree $k\ko{R}$ fails
to meet the \VW-property if and only if
$(M+\frac{1}{\km}\ko{R})\cap Z_{k\ko{R}}$ or,
equivalently,
$M\cap (Z_{k\ko{R}}-\frac{1}{\km}\ko{R})$
has points outside the diagonal
$\Q\cdot\ko{R}$.
\\[0.5ex]
First, we check that 
$M\cap Z_{k\ko{R}}\subseteq M\cap Z_{(a_\ci-1)\ko{R}}$ 
(i.e.\ without the translation) is always contained in the diagonal.
If not, then we could find $r^i,r^j\in E$ with, w.l.o.g., $i<\ci$ such that
$(a_\ci-1)r^\ci-(r^i+r^j)\in\sigma\dual$. This implies $\ci<j$,
and we choose an element $\sigc\in\innt\sigma$ such that 
$\langle\sigc, r^i\rangle = \langle\sigc, r^j\rangle \,(>0)$.
Since $r^1,\ldots,r^e$ run along the boundary of the convex polygon
$\conv(\sigma\dual\cap M\setminus 0)$, it follows that
$r^{\ci-1}, r^\ci, r^{\ci+1}\in\conv\{r^i,r^j\}$.
Thus
$$
\textstyle
\langle \sigc,\, \frac{r^{\ci-1}+r^{\ci+1}}{2}\rangle
\leq
\langle \sigc,\, r^i\rangle = \langle \sigc,\, r^j\rangle =
\langle \sigc,\, \frac{r^{i}+r^{j}}{2}\rangle,
$$
and we obtain a contradiction via
$$
\textstyle
\langle \sigc,\, r^{\ci-1}+r^{\ci+1}\rangle \leq
\langle \sigc,\, r^{i}+r^{j}\rangle \leq
\langle \sigc,\, (a_\ci-1)r^\ci\rangle <
\langle \sigc,\, a_\ci r^\ci\rangle =
\langle \sigc,\, r^{\ci-1}+r^{\ci+1}\rangle.
$$
Hence, $({M}\cap Z_{k\ko{R},\,-1})\setminus\Q\ko{R}\,$ is
non-empty if and only if 
$Z_{k\ko{R},\,-1}=Z_{k\ko{R}}-\frac{1}{\km}\ko{R}\,$ 
contains an ${M}$-lattice
point on the boundary $\partial\sigma\dual-\frac{1}{\km}\ko{R}\,$.
So we have to determine the smallest $\lambda\in\Q_{>0}$ such that 
$$
\textstyle
\lambda\cdot[\km,-g]-[0,\frac{1}{\km}]\in M
\hspace{1.0em}
\mbox{(and similarly with $[-\km,h]$).}
$$ 
This condition is equivalent to $\lambda\km\in\Z$ 
and $\lambda\cdot g + \frac{1}{\km}\in\Z$, i.e.\
$\km|(\lambda m\cdot g+1)$. By Proposition \ref{prop-abcInv},
this means $\lambda\km=c$. Hence, the first lattice point on the
shifted ray $\Q_{>0}\cdot r^1-\frac{1}{\km}\ko{R}\,$
is $[c,\,-\frac{gc+1}{\km}]$. Its value under $\sigb$ is
$$
\textstyle 
\langle\sigb,\,[c,\,-\frac{gc+1}{\km}]\rangle =
\langle (h,\km),\,[c,\,-\frac{gc+1}{\km}]\rangle =
c(h-g)-1.
$$
This leads to the condition $c(h-g)-1<k\km-1$ for
$Z_{k\ko{R},\,-1}$-membership. Thus, the \VW-condition 
coming from the ray $r^1$ is exactly the
opposite, namely $\,k\cdot\km\leq c\cdot(h-g)$. 
Similarly, the first lattice point on the
shifted ray $\Q_{>0}\cdot r^e-\frac{1}{\km}\ko{R}\,$
is $[-c',\,\frac{hc'-1}{\km}]$.
It leads to the inequality $\,k\cdot\km\leq c'\cdot(h-g)$.
\end{proof}

\subsection{Comparison of \qG- and \VW-deformations}
\label{compVWqG}
In \cite[Definition 3.7]{KSB}
the so-called T-singularities are defined as those cyclic quotient 
singularities that admit a $\Q$-Gorenstein one-parameter smoothing.
Their toric characterization can be found in
\cite[(7.3)]{minkDef} and \cite[(1.1)]{PResol}:
The toric variety $\toric(\sigma)$ is a T-singularity with Milnor number
$\mu$ if and only if $\sigma$ is the cone over a rational interval 
of integral length $\mu+1$ placed in height one.
\\[1ex]
Since an integral length does automatically imply the uniform denominator
property of (\ref{coneInterval}), this description of T-singularities 
can directly be compared to our Proposition~\ref{prop-finalQG}. 
Looking at $k=1$, it implies
that $S=\toric(C(I))$ allows a \qG-deformation at all
if and only if $|I|\geq 1$.
Altogether, we obtain the following chain of properties of 
an interval $I\not\cong[0,1]$ with uniform denominators:
\newcommand{\abst}{0.5em}
$$
(|I|=1) \hspace{\abst}\theen\hspace{\abst}
(|I|\in\Z_{\geq 1}) \hspace{\abst}\theen\hspace{\abst}
(|I|\geq 1)\hspace{\abst}\theen\hspace{\abst}
(\innt(I)\cap\Z\neq\emptyset)
$$
translating into
$$
(\mbox{T$_0$-singularity})\then
(\mbox{T-singularity})\then
(\exists \mbox{ \qG-deformation})\then
(\mbox{grounded CQS}).
$$

\subsection{The last deformation}
\label{lastDef}
Let $I=[\frac{g}{\km},\frac{h}{\km}]=[-\Ia,\Ib]$ 
\vspace{0.5ex}
be a grounded interval as in (\ref{scanIII}).
By Proposition \ref{prop-etaci}
we know that $|I|\geq a_\ci-2$ (with equality exactly for the
T-singularities).
Hence, Proposition~\ref{prop-finalQG} implies that
all subspaces $\ko{R}^\bot\subseteq \TT(-k\ko{R})$ with
$k=1,\ldots,a_\ci-2$ are \qG-deformations (hence \VW-deformations, too).
\\[0.5ex]
We will call the remaining deformation in degree $-(a_\ci-1)\cdot \ko{R}$
the ``{\em last deformation}''. This is the only degree where \qG- and 
\VW-deformations might differ at all. 
Note that the last deformation might also be the first one, i.e.\
$k=1$. This happens if and only if $a_\ci=2$, i.e.\ if and only if $\,0<\Ia,\Ib<1$.
\\[1ex]
For the following theorem, we will denote by $\{C\}:=C-\rounddown{C}$
the fractional part of a (positive, rational) number $C$.
Recall that $\Ia,\Ib\in\Q_{>0}$.

\begin{theorem}
\label{th-lastDef}
The one-dimensional subspaces $\ko{R}^\bot\subseteq \TT(-k\cdot\ko{R})$
for a grounded $S=\toric(C(I))$
with $k=1,\ldots,a_\ci-2=\rounddown{\Ia}+\rounddown{\Ib}$ are \qG- and
\VW-deformations. 
Moreover, the ``last'' deformation from
$\ko{R}^\bot$ in degree $-k\cdot\ko{R}$ with
$k=a_\ci-1$ is a
\qG- or \VW-deformation in the following cases:
\\[0.5ex]
{\rm 1)} 
The last deformation of $S=\toric(C(I))$
is \qG\ if and only if $\,\{\Ia\}+\{\Ib\}\geq 1$.
\\[0.5ex]
{\rm 2)} 
If $\{\Ia\}, \{\Ib\}\neq \frac{1}{\km}$, 
then the last deformation is \VW.
\\[0.5ex]
{\rm 3)} Otherwise,
i.e.\ if $\{\Ia\}=\frac{1}{\km}$ or $\{\Ib\}=\frac{1}{\km}$,
then the last deformation is \VW\ if and only if it is \qG. 
Hence, every \VW-deformation is \qG\ in this case.
\end{theorem}

\begin{proof}
(1) By 
Proposition \ref{prop-etaci} and \ref{prop-finalQG}, 
both sides are equivalent to $|I|\geq a_\ci-1$.
\\[1ex]
(3) The condition $\{\Ia\}=\frac{1}{\km}$ means $g\equiv -1$ (mod $\km$),
and since $c\cdot (-g)=1$ in $(\Z/\km\Z)^*$, this translates into $c=1$.
Similarly, $\{\Ib\}=\frac{1}{\km}$ is equivalent to $c'=1$.
Thus, the bounds in Proposition \ref{prop-finalQG} and \ref{prop-finalVW}
coincide.
\\[1ex]
(2)
We distinguish two cases. First, if $a_\ci\geq 3$,
then $|I|\geq 1$.
Hence
$$
a_\ci -1 \leq |I|+1 \leq \min\{c\cdot |I|, \;c'\cdot |I|\}
\;\mbox{ since } c,c'\geq 2.
$$
Otherwise, if $a_\ci=2$, then $c\geq 2$ together with 
$c\cdot (-g)\equiv 1\, (\km)$ implies that
$$
c\cdot (h-g)\geq c\cdot (-g)\geq \km+1 > \km,
$$
hence $\,c\cdot |I|>1=a_\ci-1$. Similarly we use 
$c'\cdot h \equiv 1\, (\km)$ to obtain
$\,c'\cdot |I|>a_\ci-1$.
\end{proof}

\subsection{Proof of Theorem \ref{V-VW.thm.1}}
\label{proofIntro}
We are going to proof Theorem \ref{V-VW.thm.1} of the introduction.
\\[1ex]
(1) Since $b=\gcd(n,\, q+1)$, the assumption implies $b=1$, hence
$a=\km=n$. Thus, $|I|=\frac{1}{\km}$,
and this does not leave space for $I=[\frac{g}{\km},\frac{h}{\km}]$
to become grounded, i.e.\ to allow an integer as an interior point of
$I$.
\\[1ex]
(2) Singularities admitting a \qG-smoothing are called T-singularities. In
(\ref{compVWqG}) we have seen that they correspond exactly to the
intervals of integral length, i.e.\ $\{\Ia\}+\{\Ib\}=1$. Now, the claim follows
directly from Theorem \ref{th-lastDef}\,(1).
\\[1ex]
(3) This follows because the \qG- and \VW-deformations can at most differ by
the ``last deformation''. This was just addressed in (\ref{lastDef}).
Alternatively, it follows directly from Theorem \ref{th-IAB}.

\subsection{An example of a \VW-deformation which is not \qG}
\label{exVWnotQG}
Let $\,I=[-\frac{2}{5},\frac{2}{5}]$, i.e.\ $\Ia=\Ib=\frac{2}{5}$.
\vspace{0.5ex}
This implies that $\{\Ia\}=\{\Ib\}=\frac{2}{5}$, 
i.e.\ by Theorem~\ref{th-lastDef}\,(1),
the last deformation is not \qG. Another way to see this 
is the criterion from (\ref{compVWqG}): Since $|I|<1$, there is no
\qG-deformation at all.
\\[0.5ex]
On the other hand, both $\{\Ia\}$ and $\{\Ib\}$ are different from 
$\frac{1}{5}$. Thus, Theorem \ref{th-lastDef}\,(2) implies that
the last deformation is \VW. 
Moreover, since there is no \qG-deformation at all, 
this has to be the ``first'' deformation $\ko{R}^\bot\subsetneq
\TT(-\ko{R})$ (i.e.\ with $k=1$) as well.
The other invariants are 
$$
\renewcommand{\abst}{0.7em}
n=20,\hspace{\abst}
q=11,\hspace{\abst} 
\km=a=5,\hspace{\abst}
b=4,\hspace{\abst}
\mbox{and }\; c=c'=3.
$$
The continued fraction $\frac{n}{n-q}=\frac{20}{9}$ yields
$[a_2,\ldots,a_6]=[3,2,2,2,3]$, i.e.\ $e=7$ and $\ko{R}=r^4$.
The associated $a_4=2$ was already known from our observation that the
``first'' equals the ``last'' deformation.
Finally, we obtain the following dimensions:
$$
\dim \TT=10 \mbox{ with }
\dim \TT(-k\cdot r^i)=\left\{\begin{array}{ll}
1 & \mbox{ if $k=1,2$ and $i=2,6$}\\
2 & \mbox{ if $k=1$ and $i=3,4,5$},
\end{array}\right.
$$
$\dim \TTV=3$ (degrees $-r^3,-r^4,-r^5$), and
$\dim \TTVW=1$ (in degree $-r^4$).

\subsection{Unobstructed \qG-families}
\label{unobQG}
While the focus of the paper is on the infinitesimal level,
we would just like to add how the first order \qG-deformations 
of $S=\toric(\sigma)$ extend to an unobstructed global family.
Assume that $I=[-\Ia,\Ib]$ is an interval with uniform denominators
giving rise to a cyclic quotient singularity $S_I$.
From Theorem \ref{th-IAB} we know that $d:=\dim \TTqG(S_I)$ equals
$\rounddown{\Ia+\Ib}$. This number vanishes unless $I$ is grounded.
In particular, we may write
$$
I = I' + d\cdot [0,1] 
$$
for some interval $I'$ (with uniform denominators) of length
$|I'|<1$.
In \cite[(3.2)]{flip}, such decompositions gave rise to so-called
homogeneous toric deformations of $S_I$ over the parameter space
$\A^d_k$. Its total space arises from the cone $\til{\sigma}$
taken over the 
{Cayley}-construction, i.e.\ from
$$
\textstyle
\til{\sigma}:=\Q_{\geq 0} \cdot \big(I',\,e^0\big) +
\sum_{j=1}^d \Q_{\geq 0} \cdot \big([0,1],\,e^j\big)\subseteq 
\Q\times\Q^{d+1}
$$
where $\{e^j\kst j=0,\ldots,d\}$ denotes the canonical basis of $\Q^{d+1}$.
As it is $S_I=\toric(\sigma)$, also $\toric(\til{\sigma})$ is $\Q$-Gorenstein.
The (non-toric) flat map $\toric(\til{\sigma})\to\A^d_k$ arises
from the toric map
$\toric(\til{\sigma})\to\A_k^{d+1}$ assigned to the projection
$\Z\times\Z^{d+1}\surj \Z^{d+1}$ composed with the linear projection
$\A_k^{d+1}\surj\A_k^{d+1}/k\cdot (1,1,\ldots,1)\cong\A_k^d$.
\\[2ex]
{\bf Acknowledgment:} We would like to thank F.-O.~Schreyer
for many fruitful discussions and initiating the contact on this topic.
Thanks to Jan Stevens for finding mistakes in the originally submitted
arXiv version and to the anonymous referee for valuable suggestions.

\bibliographystyle{alpha}
\bibliography{qG}

\end{document}